\documentclass[12pt]{amsart}
\usepackage{amsmath}
\usepackage{bbm}

\usepackage[T1]{fontenc}
\usepackage[bitstream-charter]{mathdesign}
\usepackage{hyperref}

\newcommand{\co}{\operatorname{Cone}}
\newcommand{\td}{\operatorname{Todd}}
\newcommand{\ict}{\operatorname{iCT}}
\newcommand{\ct}{\operatorname{CT}}
\newcommand{\ev}{\operatorname{ev}}

\newcommand{\ord}{\operatorname{ord}}
\newcommand{\SL}{\operatorname{SL}_2(\mathbb{Z})}
\newcommand{\GL}{{\rm GL}}
\newcommand{\spec}{\operatorname{Spec}}
\newcommand{\res}{\operatorname{Res}}
\newcommand{\flag}{\mathcal{F}}

\newcommand{\FA}{\mathbb{A}} 
\newcommand{\FB}{\mathbb{B}}
\newcommand{\FN}{\mathbb{N}}
\newcommand{\FZ}{\mathbb{Z}}  
\newcommand{\FC}{\mathbb{C}}  
\newcommand{\FR}{\mathbb{R}}  
\newcommand{\FQ}{\mathbb{Q}}  
\newcommand{\FP}{\mathbb{P}}

\newcommand{\bone}{{\vec{1}}}
\newcommand{\be}{\mathbf{e}}
\newcommand{\bi}{\mathbbm{i}}
\newcommand{\bj}{\mathbbm{j}}
\newcommand{\bm}{\mathbbm{m}}
\newcommand{\bp}{\mathbf{p}}
\newcommand{\br}{{\mathbbm{r}}}
\newcommand{\bx}{{\vec{x}}}

\newcommand{\fD}{\mathfrak{P}}
\newcommand{\fI}{\mathfrak{I}}
\newcommand{\fO}{\mathfrak{O}}
\newcommand{\fA}{\mathfrak{A}}

\newcommand{\fS}{\mathfrak{S}}

\newcommand{\fa}{\mathfrak{a}}

\newcommand{\fb}{\mathfrak{b}}

\newcommand{\fm}{\mathfrak{m}}

\newcommand{\sC}{\mathcal{C}}
\newcommand{\sS}{\mathcal{S}}
\newcommand{\wB}{\widetilde{B}}
\newcommand{\la}{\left<}
\newcommand{\ra}{\right>}

\newtheorem{thm}{Theorem}[section]
\newtheorem{lem}[thm]{Lemma}
\newtheorem{prop}[thm]{Proposition}
\newtheorem{defn}[thm]{Definition}
\newtheorem{cor}[thm]{Corollary}

\newtheorem{exam}[thm]{Example}
\newtheorem{remark}[thm]{Remark}

\begin{document}

\title[Rademacher function for generalized Dedekind sums]{An analogue of the Rademacher function for generalized Dedekind sums in higher dimension}

\author{Hi-joon Chae}
\address{Department of Mathematics Education, Hongik University, 
Seoul 121-791, Republic of Korea}
\email{hchae@hongik.ac.kr}
\thanks{H.C. was supported by 2013 Hongik University Research Fund.}

\author{Byungheup Jun}
\address{Department of Mathematics, Yonsei University,  
50 Yonsei-Ro, Seodaemun-Gu, Seoul 120-749, Republic of Korea}
\email{bhjun@yonsei.ac.kr}

\author{Jungyun Lee}
\address{Department of Mathematics, Ewha Womans University,
52 Ewhayeodae-gil, Seodaemun-gu, Seoul 120-750, Republic of Korea
}
\email{lee9311@kias.re.kr}



\date{May 30, 2014}

\begin{abstract}
We consider generalized Dedekind sums 
in dimension $n$,  for  fixed $n$-tuple of natural numbers, defined as sum of products of values of
periodic Bernoulli functions. 
This includes the higher dimensional Dedekind sums of Zagier and Apostol-Carlitz' generalized Dedekind sums as well as the original Dedekind sums.  
These are realized as coefficients of Todd series of lattice cones 
and satisfy reciprocity law from the cocycle property of Todd series. 
Using iterated residue formula, we compute the coefficient of the decomposition of  of the Todd 
series corresponding to a nonsingular decomposition of the lattice cone defining the Dedekind sums. 
We associate a Laurent polynomial  which is added to generalized Dedekind sums of  fixed index $\bi$ to make their denominators bounded. 
We give explicitly the denominator in terms of Bernoulli numbers.
This generalizes the role played by the rational function given by the difference of the  Rademacher function and the classical Dedekind sums. 
We associate  an exponential sum to the generalized Dedekind sums using the integrality of the generalized Rademacher function. 
We show that this exponential sum has a nontrivial bound that is sufficient to fulfill  Weyl's equidistribution criterion and
thus the fractional part of the generalized Dedekind sums are equidistributed. 
As an example, for a 3 dimensional case and Zagier's higher dimensional generalization of Dedekind sums, we compute the Laurent polynomials associated.
%
%
%
\end{abstract}

\maketitle

\tableofcontents

\section{Introduction}

\subsection{Dedekind sums and Rademacher's $\phi$-function}
Dedekind sums are rational numbers $s(a,c)$ defined for a pair of relatively prime integers $(a,c)$. 
It was introduced by R. Dedekind(\cite{Dedekind}) to describe modular transformation of his $\eta$-function:
$$
\eta(\tau) = e^{\frac{\pi i\tau}{12}}\prod_{n=1}^\infty (1-e^{2\pi i n \tau}), \quad\tau\in\frak{h}
$$
Its modular transform under the action of $A=\left(\begin{smallmatrix}a & b \\c & d \end{smallmatrix}\right)\in \SL$
is given by 
$$
\log \eta (A\tau) = \log \eta(\tau) + \frac14\log\{-(c\tau+d)^2\} + \pi i \phi(A),
$$
where $\phi : \SL \to \FQ$ is the Rademacher's $\phi$-function
\begin{equation}\label{dedekind-rademacher}
\phi\left(\begin{smallmatrix} a & b \\ c & d \end{smallmatrix} \right) =
\begin{cases}
 sign(c) \cdot s(a,c) - \frac{1}{12}\frac{a+d}{c},   &\text{if $c\ne 0$} \\
\frac{b}d,& \text{if $c=0$}
\end{cases}
\end{equation}
The Dedekind sum $s(a,c)$ is defined by above formula.

Since $\eta(\tau)$ is the 24th root of the modular discriminant $\Delta(\tau)$ which is a cusp form of weight 12, it is easy to see that $\phi(A)$ is valued in $\frac{1}{12}\FZ$.  
In other words, the Rademacher's $\phi$-function measures the failure of $\eta(\tau)$ being a modular form of weight $1/2$ and Dedekind sum  is the transcendental part of the Rademacher's $\phi$-function.

$\eta(\tau)$ has many applications in diverse disciplines of mathematics such as mathematica physics, arithmetic, geometry and low dimensional topology (eg. \cite{At}, \cite{Hir}, \cite{HZ}, \cite{K-M}, \cite{Me}, \cite{Sczech}, \cite{Siegel}).
Dedekind sums and Rademacher's $\phi$-function appear almost in the same way. 
Actually, many nontrivial properties of Dedekind sums are explained in terms of Rademacher's $\phi$-function.

It is our motivation that  partial zeta values of totally real fields at nonpositive integers have expression involving Dedekind sums and their generalization. 
The partial zeta function of an ideal $\fb$ of a number field $K$ is defined as
$$
\zeta(s,\fb) = \sum_{\fa\sim\fb} N\fa^{-s}, \quad\text{$Re(s) >1$}
$$
where $\fa$ runs over integral ideals equivalent to $\fb$. 
It is well-known that this function has a meromorphic continuation to entire complex plane admitting only a simple pole at $s=1$. 
The partial zeta function of an ideal is invariant in the class.
The sum of the partial zeta functions over the class group of a number field is  the Dedekind zeta function.

For totally real fields, it is a celebrated theorem of Klingen-Siegel (\cite{Siegel}) that the values $\zeta(1-n,\fb)$ for an ideal $\fb$ of a totally real field $K$ is a rational. 
Let us first restrict our interest on real quadratic fields. 
An ideal $\fb$ can be chosen in its class in such a way that $\fb^{-1}=[1,\omega]$, where $\omega$ is reduced element in the sense of Gauss(i.e. $\omega>1$ and $0<\omega'<1$. Here $\omega'$ denotes the conjugate of $\omega$.).
 Equivalently, $\omega$ has purely periodic negative continued fraction expansion: 
$$
\omega=[[b_0,b_1,\ldots,b_{\ell-1}]]:=b_0 -\cfrac{1}{b_1-\cfrac{1}{\cdots-\cfrac{1}{b_{\ell-1}-\cfrac{1}{\omega}}}}
$$

In \cite{Me}, a theorem of C. Meyer  tells integrality of the partial zeta values at $s=0$. Namely, 
\begin{equation}
 \zeta(0,\fb) = \frac{1}{12} \sum_{i=0}^{\ell-1} (b_i -3).
\end{equation}
On the other hand, Siegel obtains another formula for partial zeta values in terms of Dedekind sums and their generalization. 
In particular,  
for $s=0$ we have
 \begin{equation}
  \zeta(0,\fb) = s(a,c) - \frac{a+d}{12c}
\end{equation}
 where $\left(\begin{smallmatrix} a & b \\ c & d\end{smallmatrix}\right)$
 is the matrix representing the multiplication by the totally positive fundamental unit $\epsilon$ w.r.t. the basis $[1,\omega]$. 
Meyer's theorem is obtained directly by evaluating  $s(a,c)$ using the reciprocity law of Dedekind sums.
Since Dedekind sum is highly nonintegral, it is not apparent to deduce the integrality from Siegel's theorem unlike Meyer's. 
Actually, fractional part of Dedekind sums are equidistributed on the unit interval(cf. \cite{J-L2}, \cite{Myerson}, \cite{Vardi}).

However, if one notices that Siegel's theorem reads simply from \eqref{dedekind-rademacher}
$$
\zeta(0, \fb) = -\phi\left( \begin{smallmatrix} a & b \\ c & d \end{smallmatrix}\right),
$$
then the integrality follows trivially from that of the $\phi$-function.
 
\subsection{Generalization of Dedekind sums in higher degree}
Dedekind sums have generalization by taking periodic Bernoulli function $\tilde{B}_i(x)$ of higher degree instead of 
$((x))=\tilde{B}_1(x)$. 
For $i,j\ge 1$ and $a,c$ relatively prime, we define 
$$
s_{ij}(a,c) := \sum_{k=0}^{c-1} \tilde{B}_i(\frac{k}{c})\tilde{B}_j(\frac{ak}{c}).
$$
These sums are introduced through works of Apostol and Carlitz in study of modular transformation of certain Lambert series(cf. \cite{Ap}, \cite{Carlitz}).
Clearly, these sums are rational.
$i+j$ is called the \textit{weight} of $s_{ij}(a,c)$. 
It is not difficult to see that these sums survive only if the weight is even(Prop.\ref{odd-degree-sum}. See also Cor. 4.2. of  \cite{J-L2}).

For an ideal $\fb$ of a real quadratic field, Siegel gives an explicit formula of $\zeta(1-n,\fb)$ for $n\in \FN$ in terms of
the higher degree generalization of Dedekind sums of weight $2n$ in \cite{Siegel}.
Thus the rationality of $\zeta(1-n,\fb)$ is achieved automatically from that of Dedekind sums. 
What about the integrality? Again the integrality is not clear at all from Siegel's
formula. However there is still similar integrality result that
the denominator of $\zeta(-n,\fb)$ is given independently of $\fb$ investigated through works of many authors(cf. \cite{Coates-Sinott}, \cite{DR}, \cite{J-L4}, \cite{GP}, \cite{Stevens}, \cite{Zagier2}).

For example, one can check directly the integrality from the following formula of Zagier(\cite{Zagier2}):
$$
\zeta(-n,\fb)=\sum_{k=1}^r\sum_{s=0}^{2n} d_{s,n}^{(k)}\Big{(}\frac{B_{2n+2}}{2n+2}\frac{b_k^{2n-s+1}}{2n-s+1}-\frac{B_{s+1}}{s+1}\frac{B_{2n-s+1}}{2n-s+1}\Big{)},
$$
where $d_{s,n}^{(k)}$ is a coefficient of certain power of certain quadratic form with integer coefficient
determined by $\fb$ and $b_k$ is the $k$-th element of the continued fraction of $\omega$.

This is done by nonsingular decomposition of Shintani cone of the ideal and by explicitly writing
the generalized Dedekind sums. 
Later in this article, we emphasize that this is again the consequence of relating a higher degree analogue 
of the Rademacher's $\phi$-function to generalized Dedekind sums. 
It is worth to note that the same reasoning explains the fractional part of $R_{i+j} q^{i+j-2} s_{ij}(a,c)$ is equidistributed
for some integer $R_{i+j}$ determined by the weight(\cite{J-L2}). 

\subsection{Cocycle property}
These explicit formulae are all involving the terms of continued fraction of a reduced element representing the ideal.
The classical Dedekind sum and the Rademacher's $\phi$-function can be recovered from the area cocycle or the signature cocycle(eg. \cite{Asai}, \cite{K-M}, \cite{Sczech}). 
These are cocycles defined for $\SL$ and a continued fraction can be taken as sequence of $\SL$ moves.
A continued fraction is a particular nonsingular decomposition of a lattice cone in $\Lambda=\FZ^2$.
A lattice cone corresponds to a 1-simplex in $\FZ^2$, whose 0-faces are primitive lattice vectors. 
So it is natural to reconstruct the reciprocity and other properties of classical and generalized Dedekind sums from the cocycle property over singular complex consisting of lattice cones in $\FZ^2$.
Since the slope of a lattice vector corresponds to a cusp of $\frak{h}$ the upper-half plane, one can consider these cocycles defined for modular symbols of Manin(\cite{Manin}) and Stevens(\cite{Stevens}). 
Similar approaches are taken in papers by Solomon and  Sczech thru diverse context(\cite{Solomon},\cite{Sczech}).
The singular cocycle is obtained by assigning Todd power series in 2 variables  to a lattice cone. 
It is a 2-variable generalization of the classical Todd series which generate the Bernoulli numbers (up to sign): 
$$
\td(z): = \sum_{i=0}^{\infty} \frac{B_i}{i!}(- z)^i = \frac{z}{1-e^{-z}},\quad |z|<2\pi.
$$
Replacing $z$ with $\partial_z$, one obtains a differential operator  of infinite order which gives the Euler-Mclaurin summation formula.
For lattice cones, Todd series  is generalized to have several variables where the number corresponds
to the rank of the lattice, in such a way to yield the Euler-Maclaurin formula for a lattice polytope in higher dimension
in effort to count the number of geometric quantizations.  
The definition we follow appears in a paper by Brion and Vergne(\cite{BV}). 
To be precise, we refer the reader to Sec.\ref{section:Todd}.
Then the following assignment is a $(n-1)$-cocycle:
$$
\sigma \mapsto S_\sigma = \frac{\td_\sigma}{\prod_{i=1}^n \left<\sigma_i, \left(
\begin{smallmatrix} 
x \\  y 
\end{smallmatrix}
\right)
\right>}\circ{\sigma^{-1}}
$$
The cocycle condition is nothing but additivity w.r.t. barycentric decomposition of lattice cones(See \cite{Pom2}). 
For 2-dimensional cones, barycentric decomposition coincides with concatenation of two lattice cones sharing a ray.

The Todd series in 2 variables generates the Dedekind sums of higher degree as well as the classical ones.
If $\sigma$ is equivalent to the cone generated by $(1,0)$ and $(p,q)$ by change of basis of $\FZ^2$(indeed, any two 
dimensional lattice cone can be made so),  
$$
\td_{\sigma} (x,y) = \sum_{i,j\ge 0} \frac{t_{ij}(\sigma)}{i!j!} x^i y^j
$$
where 
$$t_{ij}(\sigma) = 
\begin{cases}
-(-q)^{i+j-1}\left( s_{ij}(p,q)+B_i B_j\right), & \text{if $i=1$ or $j=1$}\\
- (-q)^{i+j-1}s_{ij}(p,q), & \text{otherwise}.
\end{cases}
$$ 
In this context, 
the classical reciprocity formula for swapping $a$ and $c$ is nothing but writing down the cocycle condition for the decomposition of
the 1st quadrant as lattice cone generated by $(1,0)$ and $(0,1)$  into two by putting the lattice vector $(a,c)$.

\subsection{Distribution of Dedekind sum}
It is Rademacher who posed a question about distribution of Dedekind sums(\cite{RG}). 
In \textit{loc. cit.}, 
it is asked if the set $\left\{(\frac{p}{q}, s(p,q))\in \FR^2| (p,q)=1 \right\}$ is dense in $\FR^2$. 
The density result is proved by Hickerson(\cite{Hi}). 
Much later, Vardi  in \cite{Vardi} proves for any nonzero real $\kappa$ the fractional part of $\{\kappa \cdot s(a,c)\}$ are equidistributed on the unit interval $I=[0,1)$ in the sense of H. Weyl(\cite{Weyl}). 
In \cite{Myerson}, Myerson shows that the fractional part of  $\{(\frac{p}{q}, s(p,q))\in \FR^2| (p,q)=1 \}$ is equidistributed in $I\times I$ using similar method as Vardi. 
They identify exponential sum of Dedekind sums with (generalized) Kloosterman sums, which has a sufficiently good bound of Weil type thanks to a work of Selberg(\cite{Selberg}). 
This fulfills the Weyl's criterion for equidistribution. 
In particular, for $\kappa=12$, it is easily done by the integrality of the Rademacher's $\phi$-function. 
From 
\begin{equation}\label{rademacher_phi}
\phi\left( \begin{smallmatrix}a & b \\ c & d \end{smallmatrix} \right) = s(a,c) - \frac{a+d}{12 c} \in 
\frac{1}{12}\FZ,
\end{equation}
we have identified the exponential sum of Dedekind sums with the Kloosterman sum:
\begin{equation}
\sum_{a \in (\FZ/c\FZ)^*} \exp\left(2\pi i \left(12 s\left(a,c\right)\right)\right) = \sum_{a\in (\FZ/c\FZ)^*}\exp\left(\frac{2\pi i}c \left(a  + a^{-1}\right)\right).
\end{equation}

We emphasize that this has the same origin as the previously mentioned integrality result of the partial zeta values at $s=0$ of Meyer. 
Here  $\kappa=12$ turns out to be the universal denominator of partial zeta values of ideals of real quadratic fields for $s=-1$.

In a recent work of Jun-Lee(\cite{J-L2}), 
they extend the universal denominator to Dedekind sums generalized to higher degree.
Namely, for even integer $N=i+j$ the weight of generalized Dedekind sums,  there exists a certain integer $R_N$ determined by 
$N$ such that
\begin{equation}\label{high_integral}
c^{i+j-1} s_{ij}(a,c) - \frac{\alpha_N r_N}{R_N} \left({N-1\choose{i}} a^i + {{N-1}\choose{j}} a'^j\right) \in \frac{1}{R_N} \FZ.   
\end{equation}
Here $a'$ is a multiplicative inverse of $a$ modulo $c$. $\alpha_N$, $r_N$ are integers given by $N$. 
A representative of $a'$ in $\FZ$ uniquely determines another integer $b$ such that $a a' - bc=1$. 
One may take the formula $\eqref{high_integral}$ as definition of $\phi_{ij}(A)$: ``higher degree generalization of the Rademacher's $\phi$-function''
for $A=\left(\begin{smallmatrix} a & b \\ c & d  \end{smallmatrix}\right)\in \SL$.
%
%
Again we have a  formula  analogous to \eqref{rademacher_phi} and can associate an exponential sum similar to the Kloosterman sum to higher degree Dedekind sums 
as follows:
\begin{equation}
\begin{split}
\sum_{a \in (\FZ/c\FZ)^*}& \exp\left(2\pi i (R_N c^{N-1} s(a,c)\right)\\
 =& \sum_{a\in (\FZ/c\FZ)^*}\exp\left(2\pi i \alpha_N r_N 
\left({N-1 \choose i} a^i   + {N-1 \choose j} a'^{-1}\right)\right).
\end{split}
\end{equation}
Thanks to work of Denef and Loeser on the weight of the $\ell$-adic
cohomology giving the exponential sum(\cite{DL}), this sort of exponential sum shas good Weil bound.
Again the equidistribution in $[0,1)$ of the fractional part of   $R_N c^{N-1} s(a,c)$ turns out to be a consequence of the 
integrality as \eqref{high_integral}.

\subsection{Higher dimensional generalization}
In this paper we are interested in higher dimensional generalization of Dedekind sums as well as 
the higher degree generalization, aiming to study partial zeta values of totally real fields. 
We will investigate the relevant integrality 
as we see from the coefficients of Todd series in 2 variables.  
Recently, in \cite{C-D}, similar line of integrality is investigated by  Charollois and Dasgupta: they show the integrality of 
$\ell$-smoothed version of higher dimensional Dedekind sums on the way to showing that of $\ell$-smoothed partial zeta values at negative integers for totally real number fields.


Higher dimensional Dedekind sums arised first in topological situation. 
The following cotangent sum  associated to a lattice vector $(p_1,\ldots,p_{n-1},q)$ such that $(p_i,q)=1$ for $i=1,2,\ldots,n-1$:
\begin{equation}\label{tri}
d(q;p_1,p_2,\ldots,p_{n-1}) = (-1)^{\frac{n}2}\sum_{k=1}^{q-1} \cot \frac{\pi k }{q} \cot\frac{\pi p_1 k }{q}\cdots\cot\frac{\pi p_{n-1}k}{q}.
\end{equation}
appear as the defect of Hirzebruch's signature formula for a manifold with a finite group action on it(cf. \cite{Hir}, \cite{A-Singer}).

This trigonometric sum is called higher dimensional Dedekind sum by Zagier(\cite{Zagier1}). 
For even $n$, the rationality is obvious. Also this sum vanishes for $n$ odd due to oddity of cotangent function.
Their arithmetic properties especially the bound for denominators are studied in detail in \textit{loc. cit.}.
If $n=2$, this sum is identified with the classical Dedekind sum here in the following way:
$$
d(q;p)=4q\cdot s(p,q)
$$

Using the periodic Bernoulli function $\tilde{B}_1(x) = ((x))$, we can relate cotangent sums to higher dimensional generalization of Dedekind sums as defined below:
\begin{equation}\label{Zagier-Dedekind_sums}
\begin{split}
 s_{1,\ldots,1} (q;p_1,\ldots,p_{n-1}) &:= \sum_{ k_1,\ldots,k_{n-1}\in \FZ/q\FZ} ((\frac{k_1}{q} ))\cdots(( \frac{k_{n-1}}{q} ))(( \frac{\sum_{i=1}^{n-1}p_i k_{i}}{q} )) \\
 & =   \frac{(-1)^{\frac{n}2 + 1}}{2^n q}d(q;p_1,\ldots,p_{n-1})
\end{split}
\end{equation}
The subscript $1,1,\ldots,1$ will be justified soon below.

Replacing $\tilde{B}_1(x)$ with $\tilde{B}_i(x)$, we define  its higher degree generalization. 
Generalized Dedekind sums in higher dimension are defined as follows:
\begin{defn}\label{definition-s}
Let $(i_1,\ldots, i_n)\in \FN^n$ and $(p_1,\ldots,p_{n-1}, q)\in\FZ^n$ such that $(p_i,q)=1$ for every $i$.
The generalized Dedekind sum of $((p_1,\ldots,p_{n-1}, q)$ of index $(i_1,\ldots, i_n)$ is a rational number
$$
s_{i_1,\ldots, i_n} (q;p_1,\ldots, p_{n-1}) 
:=  \sum_{ k_1,\ldots,k_{n-1}\in \FZ/q\FZ} \tilde{B}_{i_1}(\frac{k_1}{q})\cdots\tilde{B}_{i_{n-1}}( \frac{k_{n-1}}{q} )\tilde{B}_{i_n}( \frac{\sum_{i=1}^{n-1}p_i k_{i}}{q} ).
$$
\end{defn}

These sums are recovered as coefficient of the Todd series of a lattice cone in higher dimension. 
For precise definition, we refer the reader to \S 2.
The lattice cone is given by $e_1,\ldots, e_{n-1}, (p_1,\ldots,p_{n-1},q)$. 
Let us  denote by $\td_{(q;p_1,\ldots,p_{n-1})}(x_1,\ldots,x_n)$
the corresponding Todd series: 
$$
\td_{(q;p_1,\ldots,p_{n-1})} (x_1,\ldots,x_n) := \sum_{i_1,\ldots,i_n\in \FN} \frac{t_{i_1,\ldots,i_n} (q;p_1,\ldots,p_{n-1}) }{i_1!\cdots i_n!} x_1^{i_1}\cdots x^{i_n}_n
$$
Then 
$$
t_{i_1,\ldots,i_n} (q;p_1,\ldots,p_{n-1})  = s_{i_1,\ldots,i_n} (q;p_1,\ldots,p_{n-1}) ,\quad\text{for $i_1,\ldots,i_n>1$}.
$$
If some $i_k=1$, this should be corrected by generalized Dedekind sums in lower dimension(See Thm.\ref{relation-t-s}).
Again their reciprocity law(eg. \cite{FY}, \cite{Zagier1}) is a consequence of the cocycle property of the Todd series.

\subsection{Main result}
The main result in this paper is to find the fractional part of the generalized Dedekind sums in higher dimension so that we have analogue of the Rademacher's function in full generality. 
We obtain a formula of the form of \eqref{high_integral} relevant to $s_{i_1,\ldots,i_n}(q;p_1,\ldots,p_{n-1})$.
%
As earlier, it is given as the difference of generalized Dedekind sum and  a certain rational function.
The difference  turns out to have  bounded denominator  depending not on the argument but on the weight only
as we see $12$ from the classical Dedekind sums and the Rademacher's $\phi$-function.
We are going to compute explicit bound for the denominator
 of the difference for the generalized Dedekind sums of arbitrary index.
Namely, the denominator $d_N$ is given by the formula:
\begin{equation}
d_N := \underset{\substack{m_1+\cdots+m_n = N\\ m_1 ,\ldots, m_n \ge 0}}{lcm}\left\{\text{denominator of $\prod_{i=1}^n \frac{B_{m_i}}{m_i}$} \right\}
\end{equation}
Then $d_N$ fits into our main theorem:
\begin{thm}
Let $(r_1,\cdots,r_n) \in \FN^n$ and $N=\sum_{i=1}^n r_i$.
Let $(q; p_1,\cdots,p_{n-1}) \in \FZ^n$ satisfying $(p_i,q)=1$ for $i=1,\ldots,n-1$. 
Then there exists a constant $d$ depending only on $N$ (the weight) and $n$ (the dimension) such that
$s_{r_1,r_2,\cdots,r_n}(q;p_1,p_2,\cdots,p_{n-1})$ multiplied by
$d q^{N-n+1}/(r_1!\cdots r_n!)$ is an integer and we have
\begin{equation}\label{eq-intro}
\frac{d q^{N-n+1}}{r_1! \cdots r_n!}\; s_{r_1,\cdots,r_n}(q;p_1,\cdots,p_{n-1}) \; \equiv \;  
\sum_{\bm } (- d )\prod_{i=1}^{n}\frac{B_{m_i}}{m_i!}\binom{m_i -1}{r_i-1} 
p_i^{m_i - r_i} \mod q .
\end{equation}
Here the summation is  over the set of $n$-tuples $\bm = (m_1,\cdots, m_n)$
of non-negative even integers with $\sum_{i=1}^n m_i = N$ such that 
at least one of its coordinates is zero. (We have put $p_n = -1$ for ease of notation
and $B_m$ denotes the $m$-th Bernoulli number.)
\end{thm}

Basically, our method is pretty the same as in \cite{J-L2} beside some nontriviality arising in dealing with higher dimension. 
We identify the generalized Dedekind sums in higher dimension as coefficient of the Todd series of a lattice cone in a `normal form'(See Sec.\ref{section:Dedekind}).
Then we make explicit computation of the Todd series  for a nonsingular decomposition of the cone using its cocycle property. 
The normalized Todd series has poles along the hyperplanes generated by the facets. 
While we decompose the cone, the normalized Todd series written for the decomposition acquires new poles supported on the hyperplanes generated by inner facets of the decomposition. 
These poles  are `removable singularities' and will  cancel themselves. 
The contribution of the inner cones to the generalized Dedekind sum is trivial by mod $q$ reduction. 
Explicit value is obtained by computing the coefficient of the monomial of each normalized Todd series of the boundary cones of nonsingular decomposition.
As we are dealing with higher dimension,  unlike to 1-variable case, we don't have well-defined notion of residue at a point. 
But in this case, we have to replace the point with Parshin point
given by fixing the order of the coordinate hyperplanes(cf. \cite{B,Parshin,Szenes}). 
The `iterated coefficient' does not depend on this choice of Parshin point as the Todd series is meromorphic with poles along the coordinate hyperplanes. 
As we are taking the residue modulo $q$,   
the validity of iterated residue for rational functions with general commutative ring coefficient need to be discussed in App. B.  
In this way, we will prove the main theorem.

As a corollary of the main theorem, we associate an exponential sum of certain Laurent polynomial to the generalized Dedekind sums in higher dimension. 
The main theorem is rephrased as
$$
\left<\frac{dq^{N-n}}{r_1!\cdots r_n!} s_{r_1,\ldots,r_n} (q;p_1,\ldots,p_{n-1}) \right> = 
\left<\frac{1}{q} f_{r_1,\ldots,r_{n}}(p_1,\ldots,p_{n-1})\right>
$$
where $\left<t\right>= t-[t]$ and $f_{r_1,\ldots,r_{n}}(p_1,\ldots,p_{n-1})$ is a Laurent polynomial in $p_1,\ldots, p_{n-1}\pmod{q}$.

This enables us to check the equidistribution of the left hand side by estimating the exponential sum of the right hand side.
Again  it is the dimension that makes the estimate nontrivial.
We need to estimate the exponential sum of 
the Laurent polynomial obtained above, which we denote by $K(f,q)$:
$$
K(f,q)= \sum_{p_1,\ldots,p_{n-1}\in(\FZ/q\FZ)^*} \exp\left(\frac{2\pi i }q f(p_1,\ldots, p_{n-1}) \right).
$$

For $n=2$, when we find the Kloosterman sum and its generalizations, we could apply the purity theorem of Denef-Loeser(\cite{DL}, see also \cite{J-L2}).
The crucial part of applying Denef-Loeser is the nondegeneracy of the Newton polytope at infinity associated to the Laurent
polynomial, which turns out to be highly nontrivial condition to check  in higher dimension. 
Nonetheless, we have a crude estimate that is far bigger than the best possible(Prop.\ref{prop:6.4}):
$$
|K(f,p)|\le C p^{(n-1)-\frac12}\quad \text{for $p$ prime}.
$$

This estimate relies on the condition (H) above Thm.\ref{equidistribution}. 
By (H),  the nondegeneracy condition is replaced with
much milder one but checkable at a glance of the Newton polytope. 
Namely,  the condition (H) is  
the nondegeneracy of \cite{DL} in codimension $1$ for certain variable.
Fortunately, this bound turns out to be sufficient to fullfil the Weyl's equidistribution criterion for generalized Dedekind sums in higher dimension(Thm.\ref{equidistribution}):
\begin{equation*}
\lim_{x \rightarrow \infty} \frac{1}{|I_n(x)|} 
\sum_{(p_1,\cdots,p_{n-1},q) \in I_n(x)} 
\exp\left(\frac{2\pi i k}{q}f(p_1,\cdots,p_{n-1})\right) = 0 .
\end{equation*}
Here, $I_n(x)$ is the set $\{(p_1,\ldots,p_{n-1}, q)\in \FN^n| p_i < q, (p_i,q)=1, q\le x \}$.

As the Laurent polynomial $f_{r_1,\ldots,r_n}(p_1,\ldots,p_n)$ associated to 
\begin{equation}
\frac{d q^{N-n}}{r_1! \cdots r_n!}\; s_{r_1,\cdots,r_n}(q;p_1,\cdots,p_{n-1}).
\end{equation}
fulfils the condition (H),
\textit{a priori}, the equidistribution theorem (Thm.\ref{thm-equidist}) is obtained.

A particular case tells that the Zagier-Dedekind sums of \eqref{Zagier-Dedekind_sums}, if nontrivial, after multiplication of some integer, are equidistributed in the unit interval when we take the fractional part.


This paper is composed as follows: The definition of Todd series of lattice cones and the formulation of Todd cocycle 
are given in \S 2. A precise relation between coefficients of Todd series and generalized 
Dedekind sums are given in \S 3, which will be used in the subsequent sections 
to deduce properties of latter inductively from those of former. The integrality of Todd 
coefficients and generalized Dedekind sums are shown in \S 4. 
A formula for reduction mod $q$ of generalized Dedekind sums is given in \S 5.
In \S 6, we prove the equidistribution of fractional parts of generalized Dedekind sums by estimating the exponential sum
of associated Laurent polynomial.
Finally,  in \S 7 we write explicitly the Laurent polynomials for two cases: a case of generalized Dedekind sums in 3-dimension and Dedekind-Zagier sums. 

\section*{Notations}
\begin{itemize}
\item $\bi! := (i_1!)\cdot(i_2!)\cdots(i_n!), \;\;
\bx^\bi := x_1^{i_1} x_2^{i_2}\cdots x_n^{i_n}$ 
for $\bi=(i_1,i_2, \ldots, i_n) \in \FZ_{\geq 0}^n$.
\item $|\bi| := i_1 + i_2 + \cdots + i_n,\; \; \bone := (1,1,\cdots,1)$.
\item $\la t\ra:= t - [t]$,  the fractional part of $t$.
\item $I_n :=\{ \; (q;p_1,\cdots,p_{n-1}) \in \FZ^n_{>0} \; | \;
p_1,\cdots,p_{n-1} < q \ \textrm{relatively prime to}\  q \}$.
\item $C(q;p_1,\cdots,p_{n-1})$ : cone corresponding to 
$(q;p_1,\cdots,p_{n-1}) \in I_n$ (Ex.\ref{exam-dual-cone}).
\item $M_C, \; \Lambda_C,\; \Gamma_C,\; \chi^C_i,\; P_C$
for a lattice cone $C$ (\S \ref{lattice-cone}).
\item $s_{r_1,r_2,\cdots,r_n}(q;p_1,p_2,\cdots,p_{n-1})$ : 
generalized Dedekind sum (Def.\ref{definition-s}).
\item $t_{r_1,r_2,\cdots,r_n}(q;p_1,p_2,\cdots,p_{n-1})$ : 
Todd coefficient (Def.\ref{definition-t}).
\item $\td_C(x_1,\ldots,x_n)$ : Todd series of a cone $C \subset \FR^n$
(Eq.\eqref{defn-todd}).
\item $\td_C^N(x_1,\ldots,x_n)$ : the homogeneous part of total degree $N$ of the above.
\item $\td(x_1,\ldots,x_n), \; \td^N(x_1,\ldots,x_n)$ : 
the above two objects corresponding to a nonsingular lattice cone 
(Eq.\eqref{todd-non-singular}, \eqref{todd-N}).
\item $d_{N,n}$ : the denominator of $\td^N(x_1,\ldots,x_n)$ 
(Def.\ref{definition-d}).
\item $S_C(x_1,\ldots,x_n)$: normalized Todd series
of a cone $C \subset \FR^n$ (Def.\ref{defn-n-todd}).
\item $S_C^N(x_1,\ldots,x_n)$: the homogeneous part of total degree $N-n$ of the above.
\item $T(C), \ S(C)$ : functions given by (normalized) Todd series
(Def.\ref{todd-as-functions}).
\item $T^N(C), \ S^N(C)$ : their homogeneous part of total degree $N$ and $N-n$, respectively (Def.\ref{todd-as-functions}).
\item $S_i^N(C) , \ S_o^N(C)$ : decomposition of $S^N(C)$ corresponding to a subdivision of $C$
(Eq.\eqref{decomposition-i-o}).
\item $(\ )_B$ : Let $f$ is a rational function on a vector space $V$.
For an ordered basis $B$ of $V$,
$f_B$ denotes the rational function given in coordinates with respect to $B$.
\item Bernoulli numbers $B_k$  are fixed by the generating function
$$
B(z)= \frac{z}{e^z -1} = \sum_{k=0}^\infty \frac{B_k}{k!} z^k.
$$
\item $B_{\bi} := B_{i_1}B_{i_2}\cdots B_{i_n}$ 
for $\bi=(i_1,i_2, \ldots, i_n) \in \FZ_{\geq 0}^n$
\item Bernoulli polynomials 
are defined by the generating function
\begin{equation}\label{bernoulli polynomial}
\FB(x)(z) := \frac{ze^{xz}}{e^z -1} = \sum_{k=0}^\infty \frac{B_k(x)}{k!} z^k.
\end{equation}
$B_k(x)$ is a polynomial of degree $k$ and $B_n(0)=B_n$. 
\item 
The $k$-th periodic Bernoulli function $\wB_k(t)$ for $k\ge 0$ is  
defined to be a function on $\FR$ of period $1$ by putting the values on $[0,1)$ as
$$
\wB_k(t) =\begin{cases} B_k(t), &\text{for $t \in (0,1)$} \\
B_k(0), & \text{for $t=0$ and $k>1$}\\
0, &\text{for $k=1$ and $t=0$}.
 \end{cases}
$$
\end{itemize}


\section{Todd series}\label{section:Todd}

\subsection{Lattice cones}\label{lattice-cone}
Consider the standard lattice $\FZ^n$ in $\FR^n$. 
We will introduce the notion of lattice cones with simplicial structure.
A \textit{$m$-simplicial lattice cone} is an ordered $m$-tuple $(v_1,v_2,\ldots,v_m)$ 
of primitive lattice vectors in $\FZ^n$ such that the convex hull of $\{v_1,\ldots,v_m\}$ 
does not contain the origin.
We denote the simplicial cone
of $(v_1,v_2,\ldots,v_m)$ by $\co(v_1,v_2,\ldots, v_m)$.
Since we will deal only with lattice cones in this paper, 
we often abbreviate lattice cones to  cones. 
Nonetheless, note that many of definitions below apply to general cones 
which are not necessarily lattice cones.
The underlying topological space of $C=\co(v_1,v_2,\ldots,v_m)$
is a closed subset of $\FR^n$
$$
\left|C\right|=\left| \co(v_1,v_2,\ldots,v_m)\right|:= \FR_{\ge 0} v_1 + \ldots + \FR_{\ge 0} v_m.
$$
Note that $\left|C\right|$ is a manifold with corner and
$|C|$ does not determine $C$.
The $i$-th \textit{face} of $C$ is the $(m-1)$-simplicial cone 
$$
C(i):=\co(v_1,\ldots, \hat{v}_i, \ldots, v_m).
$$
A $m$-dimensional cone $C$ 
is said to be \textit{degenerate}(resp. \textit{nondegernate}) if 
$
\dim |C| < m$(resp. 
$\dim |C| = m$).
If $m>n$, then a $m$-simplicial cone is necessarily degenerate by dimension reason.
We have an obvious action of $g\in \GL_n(\FZ)$ on the set of lattice cones, by 
$(v_1,\ldots,v_m)\mapsto (g v_1, \ldots, g v_m)$. 
Nondegeneracy is preserved under $\GL_n(\FZ)$-action.

Let $C=\co(v_1,v_2,\ldots,v_n)$ be a nondegenerate $n$-simplicial lattice cone.
We define following objects corresponding to $C$.
\begin{itemize}
\item An $(n\times n)$ integral matrix $M_C  = (v_1 | v_2|\cdots| v_n)$
where we take $v_i$ as  column vectors in $\FZ^n$
\item A sublattice $\Lambda_C = \sum_{i=1}^n \FZ v_i$ of $\FZ^n$ 
and the quotient group $\Gamma_C = \FZ^n/\Lambda_C$
\item An $n$-tuple of  characters $(\chi^C_1,\ldots, \chi^C_n)$ on $\FZ^n$ 
(or on $\Gamma_C$):
$$ \chi^C_j(v) := \exp ( 2\pi i a_j)  \quad  \textrm{ if } v = \sum_{j=1}^n a_j v_j  $$
\item  The fundamental parallelepiped $P_C$ of the torus $\FR^n/\Lambda_C$:
$$
P_C := \left\{ \; \sum_{i=1}^n a_i v_i \; \big| \; a_i\in [0,1)\quad \text{for $i=1,\ldots,n$}\; \right\}
$$
\end{itemize}
In this notation, a simple cone is said to be \textit{nonsingular} if $|\det (M_C)|=1$ 
or equivalently $\Lambda_C = \FZ^n$. 
Note that nonsingularity is preserved  and the characters $\chi^C_i$ 
of $C$ are invariant under $\GL_n(\FZ)$-action. 
The orientation of $C$ is the sign of $\det(M_C)$.

If there appears only a single simple cone $C$, we will often abbreviate $M_C$, $\Lambda_C$, 
$\chi^C_i$ and $\Gamma_C$ to $M$, $\Lambda$, $\chi_i$ and $\Gamma$, respectively.

\subsection{Chain complex of lattice cones}\label{lattice cones}
Let $\sC_k$ be the free abelian group of lattice cones generated by 
$k$-simplicial cones in $\FR^n$.
By a $k$-dimensional lattice cone, we mean an element of $\sC_k$.
The set of lattice cones make chain complex with obvious boundary operation.
Namely, for a $m$-simplicial cone $C = \co(v_1,\ldots, v_m)$, 
its boundary is a $(m-1)$-dimensional cone
$$
\partial C := \sum_{i=1}^m (-1)^{i+1} C(i).
$$
The boundary operation extends to $\sC_\bullet = \oplus_m \sC_m$.
Then $(\sC_\bullet,\partial)$ is a chain complex:
\begin{equation*}\begin{split}
\cdots \to \sC_{k+1} \xrightarrow{\partial} \sC_{k} 
\xrightarrow{\partial} \sC_{k-1}\xrightarrow{\partial} \cdots,
\quad(\partial^2 = 0)
\end{split}\end{equation*}

Let $A$ be an abelian group. 
As in a standard text in algebraic topology, by a $k$-cocycle of simplicial cones with values
in $A$, we mean an additive functional $\Phi:\sC_k \to A$,
which vanishes on boundaries (i.e. $\Phi |_{\partial(\sC_{k+1})} = 0 $).

A subdivision of a $k$-simplicial cone $C$ by a primitive lattice vector $v$ 
means the following $k$-chain:
$$
\operatorname{sbdiv}(C,v) := C  + (-1)^{k} \partial(C,v) 
$$
Here $(C,v)$ means a $(k+1)$-simplicial lattice cone generated by the basis of $C$ and $v$. 
Thus $\Phi$ being a cocycle  is equivalent to saying that
$$
\Phi(C) = \Phi(\operatorname{sbdiv}(C,v)).
$$

It is well known that a nondegenerate $n$-simplicial lattice cone $C$ admits a subdivision
into sum of nonsingular lattice cones. In other words, applying the above procedure consecutively,
we can express $C$ in $\sC_n$ as a linear combination of nonsingular cones 
modulo $\partial(\sC_{n+1})$. We remark that this notion is more general than the usual
set theoretic subdivision. A standard procedure to obtain such a (set theoretic) subdivision 
is explained in \cite{Fulton}: If $C$ is singular (i.e. if $|\det(M_C)| > 1$), 
then the fundamental parallelepiped $P_C$ contains a nonzero (primitive) lattice vector. 
Subdividing $C$ using this vector, we obtain cones with smaller determinants. 
We repeat this until individual cones have smallest possible size.

\subsection{Dual cones}
For a nondegenerate lattice cone $C=\co(v_1,\ldots, v_n)$ in $\FR^n$, 
let us define its dual lattice cone $\check{C}=\co(u_1,\ldots, u_n)$ lying in 
$\operatorname{Hom}(\FR^n,\FR)\simeq \FR^n$. 
Geometrically, $u_i$ is given as the primitive inward normal vector to the $i$-th face 
$C(i)=\co(v_1,\ldots,\hat{v_i},\ldots,v_n)$. 
We will write the dual vectors $u_i$ as row vectors in $\FZ^n$, and similarly
we define the matrix  $M_{\check{C}}$ of $\check{C}$ 
as the $(n\times n)$-matrix whose $i$-th row is $u_i$.
It can be written as a product of a diagonal matrix 
with positive diagonal entries and $M_C^{-1}$.

\begin{exam}\label{exam-dual-cone}
Let $(q;p_1,p_2,\cdots,p_{n-1}) \in I_n$.
To identify the generalized Dedekind sums, 
we need to consider the cone $C=\co(v_1,v_2,\ldots, v_n)$ with
$v_i= e_i$ for $i=1,\ldots, n-1$ and $v_n=(p_1,p_2,\ldots, p_{n-1}, q)$
where $e_i$ is the $i$-th standard unit vector in $\FR^n$.
To fix notations for later use, let us denote this cone by
$C(q;p_1,p_2,\cdots,p_{n-1})$.
Note that $v_i$ are primitive and the generators of
the dual cone $\check{C}=\co(u_1,\ldots, u_n)$ are
\begin{equation*}
\begin{split}
u_1 &= (q, 0, 0, \ldots, 0, -p_1)\\
u_2 &= (0, q, 0, \ldots, 0, -p_2)\\
&\vdots\\
u_{n-1} &= (0,0,0,\ldots, q, -p_{n-1})\\
u_n &= (0,0,0,\ldots, 0, 1) .
\end{split}
\end{equation*}
\end{exam}

\subsection{Todd series}
Let $C = \co(v_1,v_2,\ldots,v_n)$ be a nondegenerate lattice cone in $\FR^n$.
Define the \textit{Todd series} of $C$ as 
\begin{equation}\label{defn-todd}
\td_C(x_1,\ldots,x_n):= 
\sum_{\gamma\in \Gamma_C}\prod_{i=1}^n \frac{x_i}{1-\chi^C_i(\gamma) e^{-x_i}}.
\end{equation}
Then $\td_C(x_1,\ldots,x_n)$ is holomorphic at a neighborhood of 0 in $\FC^n$.
The variables $x_1,\cdots,x_n$ in (\ref{defn-todd}) should be viewed as
coordinates with respect to $\{v_1,\cdots,v_n\}$(See \S \ref{ss-todd-cocycle} below).
For a degenerate cone $C$, $\td_C(x_1,\ldots,x_n)$ is set to be  $0$. 

The Todd series is invariant under $\GL_n(\FZ)$-action on cones
due to the invariance of the characters of the cone.
In particular, the Todd series of nonsingular lattice cones in $\FR^n$
are all equal to the \textit{Todd power series in $n$ variables}:
\begin{equation}\label{todd-non-singular}
\td(x_1, x_2,\cdots,x_n) = \prod_{i=1}^{n} \frac{x_i}{1-e^{-x_i}} 
= \sum_{\br}(-1)^{|\br|}
\frac{B_{\br}}{\br!} \bx^{\br}
\end{equation}
where the second summation is over $\br = (r_1,r_2,\cdots,r_n) \in \FZ_{\geq 0}^n$
(See \textsc{Notations}).

Since the summation of the values of $\chi_i^C$ has galois invariance, 
it is easy to see that the Taylor series of $\td_C$ has coefficients in $\FQ$.
\begin{equation*}
\td_C(x_1,\ldots,x_n) = \sum_{\br}\frac{\sS_\br(C)}{\br !} \bx^\br
\end{equation*}
with the summation over $\br = (r_1,r_2,\cdots,r_n) \in \FZ_{\geq 0}^n$.
In the following section, we will see that 
$\sS_{\br}(C)$ is closely related to the higher dimensional generalized Dedekind sums.

For a nonnegative integer $N$, let $\td_C^N(x_1,\cdots,x_n) \in \FQ[x_1,\cdots,x_n]$ 
be the homogeneous part of the total degree $N$ of  $\td_C(x_1,\ldots,x_n) $. 
It is called the \textit{$N$-th Todd polynomial} of $C$.
It is the partial sum over $|\br| = N$ 
of the above sum  and is given by 
\begin{equation*}
\td_C^N(x_1,\cdots,x_n) = \left.\frac{1}{N!}\frac{\partial^N}{\partial t^N}
\td_C(t x_1, t x_2, \cdots, t x_n) \right|_{t=0} \; .
\end{equation*}
The homogeneous part $\td^N(x_1,\cdots, x_n)$ of total degree $N$
of $\td(x_1,\cdots,x_n)$
is defined similarly. It is called the \textit{$N$-th Todd polynomial in $n$ variables}.
\begin{equation} \label{todd-N}
\td^N(x_1, x_2,\cdots, x_n) = (-1)^N\sum_{|\br| = N} \frac{B_{\br}}{\br !}  \; \bx^{\br} 
\end{equation}

\subsection{Todd cocycle}\label{ss-todd-cocycle}

\begin{defn}\label{defn-n-todd}
The  normalized Todd series of an $n$-simplicial lattice cone $C$ in $\FR^n$ is 
the meromorphic function around $0$
$$
S_C(x_1,\ldots,x_n) := \frac{\td_C(x_1,\ldots,x_n)}{(\det M_C) x_1x_2\cdots x_n}
$$
in $\FC^n$ with poles along the coordinate hyperplanes.
\end{defn}

For a nonnegative integer $N$, let $S_C^N(x_1,\cdots,x_n)$ be the homogeneous
part of total degree $N-n$. Of course, it is given by
$$
S^N_C(x_1,\ldots,x_n) = \td_C^N(x_1,\ldots,x_n)/(\det M_C) x_1x_2\cdots x_n.
$$

To deal with Todd series for various cones in $V = \FR^n$ simultaneously, 
it is necessary to view $T_C(x_1,\cdots,x_n)$ and $S_C(x_1,\cdots,x_n)$
as functions on $V$ (or on $V\otimes \FC$) by taking
variables $x_1, x_2, \cdots, x_n$ in the above definition 
as coordinates on $V$ with respect to the ordered basis $\{v_1, v_2, \cdots, v_n\}$
if $C = \co(v_1, v_2, \cdots, v_n)$ is nondegenerate. 
Let us denote these functions by $T(C)$ and $S(C)$, respectively:

\begin{defn}\label{todd-as-functions}
Let $C = \co(v_1, v_2, \cdots, v_n)$ be an $n$-simplicial nondegenerate
lattice cone in $V = \FR^n$.
Define meromorphic functions $T(C)$ and $S(C)$ on $V_{\FC} = V\otimes \FC$ by
\begin{align*}
T(C)  : x_1 v_1 + \cdots + x_n v_n & \mapsto 
\td_C(x_1,\cdots,x_n) \\
\ S(C): x_1 v_1 + \cdots + x_n v_n & \mapsto    S_C(x_1,\cdots,x_n)
\end{align*}
For a nonnegative integer $N$, the homogeneous polynomial $T^N(C)$ and
the homogeneous rational function $S^N(C)$ on $V$ are defined similarly.
\begin{align*}
T^N(C)  : x_1 v_1 + \cdots + x_n v_n & \mapsto 
\td^N_C(x_1,\cdots,x_n) \\
\ S^N(C): x_1 v_1 + \cdots + x_n v_n & \mapsto    S^N_C(x_1,\cdots,x_n)
\end{align*}

\end{defn}

We remark that $S^N(C)$ can be obtained from $S(C)$ (in a coordinate-free way):
for $v \in V$, $S^N(C)(v)$ is the coefficient  of $t^{N-n}$ of the Laurent polynomial
$S(C)(tv)$ in one variable $t$.

\begin{remark}
Let $y_1, y_2, \cdots, y_n$ be coordinates with respect to the standard basis of $V$.
Then the function $S(C)$ is given by
$ S_C((y_1,\cdots, y_n)(M_C^{-1})^T) \in \FQ((y_1,\cdots,y_n))$
in terms of these coordinates.
\end{remark}

The following proposition, which we call ``the cocycle property of Todd series'',
is a restatement of \cite[Thm.3]{Pom2} in frame of this article. 

\begin{prop}[Pommersheim]
The association $\Phi: C \mapsto S(C) $
is an $n$-cocyle of simplicial lattice cones in $V = \FR^n$ 
with values in the space of meromorphic functions on $V_{\FC}$.
\end{prop}

\begin{cor}\label{todd-poly-cocycle}
Let $N$ be a nonnegative integer.
The association $\Phi: C \mapsto S^N(C) $
is an $n$-cocyle of simplicial lattice cones in $V = \FR^n$ 
with values in the space of rational functions on $V_{\FC}$.
\end{cor}


\section{Dedekind sums and Todd coeffients}\label{section:Dedekind}

Let $(q;p_1,\cdots,p_{n-1}) \in I_n$ (See \textsc{Notations}).
Consider the cone $C= C(q;p_1,\cdots,p_{n-1})$
and  its dual lattice cone $\check{C}=\co(u_1,\ldots, u_n)$ as
given in Example \ref{exam-dual-cone}. Recall the generators of $C$ are
$v_i= e_i$ for $i=1,\ldots, n-1$ and $v_n=(p_1,p_2,\ldots, p_{n-1}, q)$.
In this case, we have 
$u_i = q v_i^*$ for $1 \leq i \leq n$ where $\{v_1^*, \cdots, v_n^*\}$ is 
the basis dual to $\{v_1,\cdots,v_n\}$ (i.e. $\left<v^*_i, v_j\right>=\delta_{ij}$).
We would like to identify the coefficient of the Todd series of $C$ 
using generalized Dedekind sums. Expanding the denominators in (\ref{defn-todd}),
we have
\begin{equation}\label{first-expansion}
\td_C(x_1,\ldots,x_n) 
= x_1 x_2 \cdots x_n \sum_{m\in\Gamma_C} \sum_{\ell_1,\ldots,\ell_n=0}^\infty \chi_1(m)^{\ell_1}\cdots \chi_n(m)^{\ell_n} e^{-\ell_1 x_1} \cdots e^{-\ell_n x_n}.
\end{equation}
This series converges absolutely for totally positive $(x_1,\ldots,x_n)$,
but can be continued analytically to a neighborhood of $0$, since
the Todd series itself is analytic at $0$. 

Notice that $m\mapsto \chi_1(m)^{\ell_1}\cdots \chi_n(m)^{\ell_n}$ 
is again a character on $\Gamma_C$. 
Let us denote this character by $\chi_{\ell_1\ldots\ell_n}$.
As 
$\chi_i(m) = \exp\left(2\pi i \langle m, \frac{u_i}q\rangle\right)
$,
we have
$$
\chi_{\ell_1\ldots\ell_n}(m)=
\exp\left(2\pi i \left< m, \sum_{i=1}^n\ell_i\frac{u_i}q\right>\right) .
$$
In the summation over $m \in \Gamma_C$ in (\ref{first-expansion}), 
we will use a common trick of exponential sums:
$$
\sum_{m\in \Gamma_C} \chi_{\ell_1,\ldots,\ell_n} (m) =
\begin{cases}
\left|\Gamma_C\right| = q, &  \text{if $\chi_{\ell_1,\ldots,\ell_n}$ is trivial} \\
0, &\text{otherwise}
\end{cases}
$$
Note $\chi_{\ell_1,\ldots,\ell_n}$ is trivial if and only if 
$\sum_{i=1}^n \ell_i \frac{u_i}q $ is a lattice vector in $\FZ^n$. 
Since we have
$$
\left\{ \sum_{i=1}^n \ell_i \frac{u_i}q \; \big| \; \ell_i \in \FZ_{\ge 0} \right\}
= \frac{1}{q} \Lambda_{\check{C}} \supset |C^\vee| \cap \FZ^n \supset \Lambda_{\check{C}},
$$
by summing over $m \in \Gamma_C$ first in (\ref{first-expansion}),
we may rewrite $\td_C$ as summation over the lattice points inside $C^\vee$:
\begin{equation}\label{second-expansion}
\td_C\left(x_1,x_2,\cdots,x_n\right) = 
qx_1x_2\cdots x_n \sum_{m\in |\check{C}|\cap\FZ^n}
e^{-\sum_{i=1}^n\langle m,v_i \rangle x_i},
\end{equation}

\begin{remark}
Similar argument shows that the above equation holds for any nondegenerate
lattice cone $C = \co(v_1,\cdots,v_n)$ if we replace $q$ in the equation
by $|\det M_C| = |\Gamma_C|$. 
\end{remark}

The right hand side of (\ref{second-expansion}) is defined 
for $(x_1,\ldots,x_n) \in \FR_{\geq 0}^n$ 
and is analytically continued to a neighborhood of $0$.
Any lattice point $u \in |\check{C}|\cap \FZ^n$ can be written uniquely as 
$u = w + i_1 u_1 + \cdots + i_n u_n$ with $w \in P_{\check{C}} \cap \FZ^n$
and $i_1, \cdots, i_n \in \FZ_{\geq o}$. The set of lattice points in the fundamental parallelepiped
for $\Lambda_{\check{C}}$ is given by
$$
P_{\check{C}}\cap\FZ^{n} = 
\left\{ \; 
\sum_{i=1}^{n-1} \frac{k_i}{q} u_i + 
\left<\frac{p_1 k_1 + \cdots + p_{n-1} k_{n-1}}{q}\right> u_n  \Big|  \quad\text{for $k_i= 0,1,\ldots,q-1$}  \; \right\} \;\; ,
$$
and we can write the right hand side of (\ref{second-expansion}) as
\begin{equation*}
\begin{split}
&= q x_1\cdots x_n \sum_{u \in {P_{C^\vee}\cap \FZ^n}} e^{-\sum_{i=1}^n \langle u , v_i\rangle  x_i}\sum_{i_1,\ldots,i_n=0}^\infty e^{-i_1 q x_1} e^{-i_2 q x_2} \cdots e^{-i_n q x_n} \\
&= q^{-n+1} 
\sum_{k_1,\ldots,k_{n-1}=0}^{q-1} \left(  \frac{q x_1 e^{-k_1 x_1}}{1-e^{-qx_1}} \right)\left(
\frac{q x_2 e^{-k_2 x_2}}{1-e^{-qx_2}} \right)
\cdots \left(\frac{q x_{n-1} e^{-k_n x_{n-1}}}{1-e^{- q x_{n-1}}}\right)
\left(\frac{q x_n e^{- \la\frac{\sum_{i=1}^{n-1} p_i k_i}{q}\ra   q x_n}}{1-e^{-q x_n}}\right) \\
&= q^{-n+1} \sum_{k_1,\ldots,k_{n-1}=0}^{q-1} 
\FB
\left(\frac{k_1}{q}\right)\left(-qx_1\right) \FB\left(\frac{k_2}{q}\right) \left(-qx_2\right) 
\cdots \FB\left(\la\frac{\sum_{i=1}^{n-1} p_i k_i}{q}\ra\right) \left(-qx_n\right)
\end{split}
\end{equation*}
where $\FB(x)(z)$ denotes the generating function of Bernoulli polynomials given 
by (\ref{bernoulli polynomial}). Expanding the above further using Bernoulli polynomials,
we obtain another expression of the Todd series
\begin{equation*}
\begin{split}
&= q^{-n+1} \sum_{\bj\in\FZ_{\ge 0}}
\sum_{k_1,\ldots,k_{n-1}=0}^{q-1}
\frac{B_{j_1}\left(\frac{k_1}q \right) \cdots B_{j_{n-1}}
\left(\frac{k_{n-1}}q\right) B_{j_n} 
\left(\la\frac{\sum_{i=1}^{n-1} p_i k_i}{q}\ra\right)}
{ \bj !
}
(-q x_1)^{j_1} \ldots (-q x_n)^{j_n}\\
&= 
\sum^\infty_{N=0} \sum_{|\bj|=N}
 \sum_{k_1,\ldots,k_{n-1}=0}^{q-1} (-1)^{N 
 } 
 q^{N -n +1 
 }
\frac{B_{j_1}\left(\frac{k_1}q \right) \cdots B_{j_{n-1}}\left(\frac{k_{n-1}}q\right) 
B_{j_n} \left(\la\frac{\sum_{i=1}^{n-1} p_i k_i}{q}\ra\right)}%
{\bj!
}
\vec{x}^{\bj}
\end{split}
\end{equation*}
whose coefficients are very closed to the generalized Dedekind sums.
Here, $\bj = (j_1,\ldots, j_n)$.

\begin{defn}\label{definition-t}
For $\bj=(j_1,\cdots, j_n) \in \FZ_{\geq 0}^n$ and $(q;p_1,\cdots,p_{n-1}) \in I_n$,
we define the Todd coefficient by
\begin{equation*}
t_{\bj 
}(q;p_1,\ldots,p_{n-1}) =
\sum_{k_1,k_2,\ldots,k_{n-1}=0}^{q-1} 
B_{j_1}\left(\frac{k_1}q \right)  \cdots B_{j_{n-1}}\left(\frac{k_{n-1}}q\right) 
B_{j_n} \left(\la\frac{\sum_{i=1}^{n-1} p_i k_i}{q}\ra\right).
\end{equation*}
\end{defn}

Thus the Todd series of $C$ is written as
\begin{equation}\label{todd-t}
\td_C(x_1,\ldots,x_n)
=  
\sum^\infty_{N=0} \sum_{|\bj|=N} (-1)^{N} q^{N-n+1}
\frac{t_{\bj}(q;p_1,\ldots,p_{n-1})}{\bj!}
\vec{x}^\bj 
.
\end{equation}
Note that $t_{\bj}(q;p_1,\ldots,p_{n-1})$ remains unchanged if we replace
$p_i$ by any number in the same congruence class modulo $q$.
Also we have the  vanishing of Todd coefficients of odd weight in the next two propositions.
These generalize Cor.4.2 in \cite{J-L2} to arbitrary dimension $n$.

\begin{prop}\label{vanishing-t}
If the total degree $N = |\bj|$ 
is odd, 
then we have
$$
t_{\bj}(q;p_1,\ldots,p_{n-1}) = 0.
$$
\end{prop}

\begin{proof}
The proof is similar to that of Cor.4.2 in \cite{J-L2}.
We can write the function appearing in the definition (\ref{defn-todd}) of Todd series as
\begin{equation*}
\frac{x}{1-\chi(\gamma) e^{-x}} = \frac{x}{2} + L^{\chi(\gamma) }(x)
\quad \textrm{where } \quad
L^{\lambda}(x) = \frac{x}{2}\cdot \frac{1+\lambda e^{-x}}{1-\lambda e^{-x}} \; .
\end{equation*}
If $\lambda \neq 1$, then $L^{\lambda}(x)$ is not an even function.
But since $L^{\lambda}(-x) = L^{\lambda^{-1}}(x)$, the sum
$\sum_{\gamma \in \Gamma_C} L^{\chi(\gamma) }(x)$ is even. 
So is the sum of products 
$\sum_{\gamma \in \Gamma_C} L^{\chi_1(\gamma) }(x_1)\cdots 
L^{\chi_k(\gamma) }(x_k)$.
By expanding the product in (\ref{defn-todd}), 
we see that the odd part of $\td_C$ is the sum of
\begin{equation*}
2^{-k}x_{i_1}x_{i_2} \cdots x_{i_k} \times 
\sum_{\gamma \in \Gamma_C}L^{\chi_{i_{k+1}}(\gamma) }(x_{i_{k+1}} )
\cdots L^{\chi_{i_n}(\gamma) }(x_{i_n})
\end{equation*}
with $k$ odd and $\{i_1,\cdots, i_n\}$ a permutation of $\{1,\cdots, n\}$. 
Notice that $k$ is the number of $x_i$'s of multiplicity $1$. Thus the odd part is supported on monomials $x_1^{j_1}\ldots x_n^{j_n}$ with some $j_k=1$. This finishes the proof.
\end{proof}

Compared to $t_{j_1,\ldots,j_n}$,
the generalized Dedekind sums $s_{j_1,\ldots,j_n}$ 
in Def.\ref{definition-s} are defined using periodic Bernoulli functions $\wB_j(t)$ 
in place of  Bernoulli polynomials $B_j(t)$.
Since $\wB_j(t) = B_j(t)$ on $[0,1)$ if $j >1$, we have
\begin{equation}
s_{j_1,\ldots,j_n}(q;p_1,p_2,\ldots,p_{n-1}) 
= t_{j_1,\ldots,j_n}(q;p_1,p_2,\ldots,p_{n-1})
\quad \textrm{if } j_1,\ldots, j_n >1 .
\end{equation}
Unlike Todd coefficients, all generalized Dedekind sums of odd total degree vanish.

\begin{prop}\label{odd-degree-sum}
If the total degree $N = j_1 + j_2 + \cdots + j_n$ is odd, then
the generalized Dedekind sum $s_{j_1,\ldots,j_n}(q;p_1,p_2,\ldots,p_{n-1})$ vanishes.
\end{prop}

\begin{proof}
We will use induction on $n$. Cor.4.2 in \cite{J-L2} is the case when $n=2$.
In Def.\ref{definition-s}, $s_{j_1,\ldots,j_n}(q;p_1,p_2,\ldots,p_{n-1}) $
is defined as a summation over the set $K$ of 
$(n-1)$-tuples of non-negative integers less than $q$.
For each $I \subset \{1, \cdots, n-1\}$, let $K(I)$ be the set of
$(k_1,\cdots, k_{n-1}) \in K$ such that 
$k_i = 0$ for $i \in I$ and $k_i \neq 0$ if $i \not\in I$.
Then $K$ is the disjoint union of $K(I)$'s. We claim that the partial sum
(of the sum in Def.\ref{definition-s}) over each $K(I)$ vanishes.

The $j$-th periodic Bernoulli function $\wB_j(x)$ is even (odd, respectively) 
if $j$ is even (odd, respectively). Hence in the summation over $K(\emptyset)$,
the terms corresponding to $(k_1,\cdots, k_{n-1})$ and
$(q-k_1,\cdots, q-k_{n-1})$ cancel each other. Suppose $I \neq \emptyset$.
If there exist $i \in I$ such that $j_i$ is odd, then the summation over $K(I)$
is zero simply because $\wB_{j_i}(0) = 0$. Otherwise, i.e. if $j_i$ is even
for each $i \in I$, then the summation over $K(I)$ is a generalized Dedekind
sum of odd total degree in fewer variables. Hence it vanishes by induction
assumption. 
\end{proof}

When $j=1$, we still have $\wB_1(t) = B_1(t)$ on $(0,1)$ 
but $\wB_1(0) = 0, B_1(0)=B_1 = - 1/2$.
Hence $t_{i_1,\ldots,i_n}$ and $s_{i_1,\ldots,i_n}$ may be different 
if $1 \in \{i_1,\cdots,i_n\}$. 
A precise relation is given inductively as follows.
Let $(i_1,i_2,\cdots,i_n) \in \FZ^n_{>0}$ and $(q;p_1,\cdots,p_{n-1}) \in I_n$.
We set $p_n = -1$. (This convention will be also used in later sections.
It makes statements, not proofs, of several theorems easier.)
Let $J = \{ \; j \; | \; i_j = 1 \; \}$.
Given a nonempty subset $T$ of $J$, let 
$\{j_1,\cdots,j_r\} = \{1,2,\cdots,n\} \backslash T$ ordered so that $j_1 < \cdots < j_r$. 
For each $j_k$, choose an integer $p_{j_k}^T$ such that
$p_{j_k}^T \equiv -p_{j_r}^{-1}p_{j_k} \mod q$. 
(Hence if $n \in \{1,\cdots,n\} \backslash T$,
then we can and will set $p_{j_k}^T = p_{j_k}$.)

\begin{thm}\label{relation-t-s}
With notations as in the above paragraph, we have
\begin{multline*}
s_{i_1,\ldots,i_n} (q;p_1,\ldots,p_{n-1})
= t_{i_1,\ldots,i_n}(q; p_1,\ldots, p_{n-1}) \\
+ \sum_{\emptyset \ne T\subset J}{\textstyle \left(\frac{1}{2}\right)^{|T|} }
 t_{i_{j_1},\ldots, i_{j_r}} (q; p^T_{j_1},\ldots, p^T_{j_{r-1}}) ,
\end{multline*}
\begin{multline*}
t_{i_1,\ldots,i_n} (q;p_1,\ldots,p_{n-1})
= s_{i_1,\ldots,i_n}(q; p_1,\ldots, p_{n-1}) \\
+ \sum_{\emptyset \ne T\subset J}{\textstyle\left(-\frac{1}{2}\right)^{|T|} }
 s_{i_{j_1},\ldots, i_{j_r}} (q; p^T_{j_1},\ldots, p^T_{j_{r-1}}) .
 \end{multline*}
The second summation is
over nonempty subsets $T \subset J$ with $|T| \equiv N \mod 2$
where $N = i_1 + \cdots + i_n$ is the total degree.
\end{thm}

\begin{proof}
Apply $\wB_j(x) = B_j(\{x\}) +\frac12 \delta_{1,j} \delta(\{x\}) $ to
the definition of $s_{i_1,\ldots,i_n}$.
We omit the details. The statement on the second summation follows 
from the last proposition. 
\end{proof}

\begin{remark}\label{defn-exceptional}
It is convenient to define Todd coefficients
and generalized Dedekind sums when $n= 0, 1$ as follows,
so that the above equations in the theorem hold.
When $n=1$, we define $t_j(q) = B_j(0) = B_j$ and $s_j(q) = \wB_j(0) $ for $j  \geq  1$.
When $n=0$, we define $t = s = 1$.
\end{remark}


\section{Integrality of  generalized Dedekind sums}

Let $V$ be a $n$-dimensional vector space over $\FR$ 
and let $S(V^*)$ and $R(V^*)$ be the symmetric algebra of $V^*$ 
and its field of fractions, respectively. These are identified with 
the ring of polynomial functions and the field of rational functions on $V$, respectively.  
If we choose an ordered basis $B$ for $V$, then they are identified with 
the sets of polynomials and rational functions in coordinates with respect to $B$.
When we need to specify the basis, we will denote by $(\;)_B$. For example, if
$v^* \in V^*$ and $B = \{w_1,\cdots,w_n\}$ is a basis for $V$, then
$(v^*)_B = \sum_{i=1}^{n} \langle v^*, w_i\rangle x_i \in \FR[x_1,\cdots, x_n]$.
 
Let $C = \co(v_1, \cdots, v_n)$ be a lattice cone in $V$ w.r.t. a lattice associated to $V$. 
So $v_i$ are primitive lattice vectors in $V$.
The functions $T^N(C) \in S(V^*)$ and $S^N(C) \in R(V^*)$
given in Def.\ref{todd-as-functions} can be written as
\begin{align*}
T^N(C) &= \td_C^N(v_1^*, v_2^*, \cdots, v_n^*) \notag \\
S^N(C) &= \frac{T^N(C)}{(\det M_C) v_1^* v_2^*\cdots v_n^*}
\end{align*}
where $\{v_1^*, \cdots, v_n^*\}$ is the basis for $V^*$ dual to $\{v_1, \cdots,v_n\}$.
For $v^* \in V^*$, we will denote the hyperplane $(v^*)^{\perp} = 
\{ v \; | \; \langle v^*, v \rangle = 0 \}$ in $V$ by $v^* = 0$ for simplicity.

Let $C = \co(w_1,\cdots,w_n) \subset V$ be a simplicial lattice cone.
The cone $C$ admits a subdivision $C = \cup_{D \in \fD}\; D$ 
into nonsingular lattice cones, all of which are of the same orientation as $C$.
 (See \S \ref{lattice cones}.)
Then we have by Cor.\ref{todd-poly-cocycle}
\begin{equation}\label{cocycle-eq}
S^N(C) = \sum_{D \in \fD} S^N(D) .
\end{equation}
In particular, the only possible singularities of this rational function are
poles along the hyperplanes $w_j^* = 0 \; (1\leq j \leq n)$.

\begin{remark}
The individual rational functions in the right hand side of (\ref{cocycle-eq})
have poles along hyperplanes generated by facets of $D \in \fD$. 
(A facet is a face of codimension 1.) The second statement says that
most of these poles cancel each other except the outermost ones. This fact can be proven
without recourse to the cocycle property.
Really, it is easy to prove the following.

Let $C$ and $D$ be two nonsingular lattice cones in $V$ having a common facet.
Suppose $C$ and $D$ are of the same (resp. opposite) orientation 
if they are in the opposite (resp. same) sides of the hyperplane $H$ 
generated by the common facet.
Then the rational function $S^N(C) + S^N(D)$ does not have a pole along $H$.
\end{remark}

From now on, we suppose that any proper subset of $\{w_1, \cdots, w_n \}$ 
can be extended to a basis of $\FZ^n$. Under this assumption,
there exists a subdivision $\fD$ (actually, one obtained by the above mentioned 
procedure in \cite{Fulton})
such that each facet of $C$ is contained in a unique $D \in \fD$, 
i.e. facets of $C$ are not subdivided. Really, the fundamental parallelepiped 
of each facet of $C$ does not contain a non-zero lattice vector.
For $1 \leq j \leq n$, let $D_j \in \fD$ be the cone in the subdivision $\fD$ 
that contains the $j$-th facet $C(j) = \co(w_1, \cdots, \widehat{w_j},\cdots,w_n)$ of $C$.
Thus we can divide the subdivision into two disjoint parts:
$$
\fD = \fO \cup \fI
$$
where $\fO = \{D_1, \cdots, D_n\} \subset \fD$ and $\fI = \fD \backslash \fO$.
Namely, $\fO$(resp. $\fI$) is the set of outer(resp. inner) cones in $\fD$. 
According to this decomposition, we decompose the sum (\ref{cocycle-eq}) into two parts:
\begin{equation}\label{decomposition-i-o}
S^N(C) = S_{inner}^N(C) + S_{outer}^N(C) ,
\end{equation}
where $S_{inner}^N(C) = \sum_{D \in \fI} S^N(D)$ and 
$S_{outer}^N(C) = \sum_{D \in \fO} S^N(D)$.

\begin{prop}\label{inner-sum}
The only possible singularities of the rational function $S_{inner}^N (C)$ are 
simple poles along the hyperplanes generated by facets of 
$D \in \fO$ except $w_j^* = 0 \;\; (1\leq j \leq n)$.
\end{prop}

\begin{proof}
We have $S_{inner}^N (C) = S^N (C) - S_{outer}^N(C)$ and the only
possible singularities of $S_{outer}^N(C)$ are poles along hyperplanes generated
by facets of $D \in \fO$. It remains to show that if $H$ is the hyperplane
$w_j^* = 0$ generated by the $j$-th facet 
$C(j) =  \co(w_1, \cdots, \widehat{w_j},\cdots,w_n)$ of $C$ for $1\leq j \leq n$, 
then $S_{inner}^N (C)$ does not have a pole along $H$.
Suppose otherwise. Then $H$ must
be generated by a facet $F$ of some $D \in \fI$. Since $D$ is in $\fI$, $F$ is
not a face of $C$ and is shared by a unique $D'\neq D \in \fD$. This implies
$D'$ and $D$ are in the opposite sides of $H$, which is a contradiction.
\end{proof}

\begin{defn}\label{definition-d}
Let $d_{N,n}$ be the least common multiple of the denominators of coefficients 
of $\td^N(x_1,\cdots,x_n)$, the $N$-th Todd polynomial (\ref{todd-N}).
\end{defn}

\begin{remark}
It is well known that the primes dividing the denominator
of the $k$-th Bernoulli number ($k$ even) are precisely 
those $p$ such that $p-1$ divides $k$.
\end{remark}

Let $(q;p_1,\cdots,p_{n-1}) \in I_n$. 
Let $E = \{ e_1, \cdots, e_n \}$ be the standard basis of $V = \FR^n$ and let 
$B = \{ w_1, \cdots, w_n \}  \subset \FZ^n$
where $w_1 = e_1, \cdots, w_{n-1} = e_{n-1}, w_n = (p_1,\cdots, p_{n-1}, q)$.
Any subset of $B$ can be extended to a basis of $\FZ^n$.
Hence the cone $C = \co(w_1, \cdots, w_n)$ (which we have denoted
by $C(q;p_1,\cdots,p_{n-1})$ in Ex.\ref{exam-dual-cone})
admits a subdivision $\fD = \fI \cup \fO$ as described  in the paragraph before
the last proposition. The coordinates with respect to the basis $B$ will be
denoted by $x_1, \cdots, x_n$. Hence $(w_j^*)_B = x_j$ for $1\leq j \leq n$.
To keep the notations simple, we will omit $(\;)_B$ and write $w_j^* = x_j$ if
it is clear from the context.  

We have the following result on the integrality of its coefficients
of $\td_C^N(x_1, \cdots, x_n) = T^N(C)_B$.

\begin{thm}\label{integrality}
For a nonnegative interger $N$ and $(q;p_1,\cdots,p_{n-1}) \in I_n$, let
$d = d_{N,n}$ and let $C = C(q;p_1,\cdots,p_{n-1})$. Then we have
$$\td_C^N(x_1,\cdots,x_n) = q x_1 \cdots x_n S^N(C)_B
\; \in \; \frac{1}{d}\FZ[x_1, x_2, \cdots, x_n] .$$
\end{thm}

\begin{proof}
We keep the notations of this section.
Recall that for $1\leq j \leq n$, $D_j \in \fO$  is the unique nonsingular lattice cone 
in the subdivision $\fD$ of $C$ which has 
$C(j)$ as a facet. 
We order the generators of  $D_j = \co(v_1^{(j)},\cdots, v_n^{(j)})$ such that 
$v_i^{(j)} = w_i$ for $i \neq j$. Since $D_j$ is a nonsingular lattice cone,
$v_1^{(j)}\wedge \cdots \wedge v_n^{(j)} = \pm e_1 \wedge \cdots \wedge e_n$. The
sign must be $+1$ since 
$w_1 \wedge \cdots \wedge w_n = q e_1 \wedge \cdots \wedge e_n$
with $q \geq 2$ and $w_j$ and $v_j^{(j)}$ are in the same side of 
the hyperplane generated by the facet $C(j)$.  Let $\{ v_1^{(j)*},\cdots, v_n^{(j)*}\}$
be the basis of $V^*$ dual to $\{v_1^{(j)},\cdots, v_n^{(j)}\}$. We claim that
$v_j^{(j)*} = q w_j^*$. By definition, $\langle v_j^{(j)*} , w_i \rangle = 0$ if $i \neq j$. 
Since $w_j = \sum_i v_i^{(j)} \langle v_i^{(j)*} , w_j \rangle$, we have
$w_1 \wedge \cdots \wedge w_n = \langle v_j^{(j)*} , w_j \rangle 
v_1^{(j)}\wedge \cdots \wedge v_n^{(j)} $. 
So $\langle v_j^{(j)*} , w_j \rangle = q$, in other words, 
$(v_j^{(j)*})_B = q x_j$ for $1 \leq j \leq n$.

If $D = \co(v_1,\cdots,v_n)$ is a nonsingular lattice cone in $V = \FR^n$, then 
both of the numerator and the denominator of 
$d S^N(D)_E = d T^N(D)_E/(v_1^* \cdots v_n^*)_E$
are integral polynomials. 
In coordinates w.r.t. $\{v_1, \cdots, v_n\}$, 
it holds by definition, and then we change the variables to coordinates with respect to $E$. 
Since $\{v_1,\cdots,v_n\}$ is a basis for $\FZ^n$, $(v_j^*)_E \;\; (1\leq  j \leq n)$ are
primitive (linear) polynomials hence are irreducible objects in the integral polynomial ring.
The terms in the right hand side of 
$$ d S^N(C)_E = \sum_{D \in \fD} d S^N(D)_E $$
are such rational functions.  
Since the integral polynomial ring is a UFD,
if we factor out the right hand side of the above equation, only 
$(v_1^{(1)*} \cdots v_n^{(n)*})_E$ is left in the denominator by (\ref{cocycle-eq}).
In other words, $d (v_1^{(1)*} \cdots v_n^{(n)*} S^N(C))_E$ is an integral polynomial.
In coordinates with respect to $B$, we have 
\begin{equation*}
d (v_1^{(1)*} \cdots v_n^{(n)*} S^N(C))_B 
= d q^nx_1\cdots x_n S^N(C)_B \in \FZ[x_1, \cdots, x_n] .
\end{equation*}
Let  $r_q: \FZ[x_1,\cdots, x_n] \rightarrow \FZ/q\FZ[x_1,\cdots, x_n]$ 
be the the reduction modulo $q$ map and let $\fS = \fS_q \subset \FZ[x_1,\cdots, x_n]$ 
be the inverse image of the set of non-zerodivisors in $\FZ/q\FZ[x_1,\cdots, x_n]$.
(A non-zero integral polynomial belongs to $\fS$ if and only if the greatest common divisor 
of its coefficients is relative prime to $q$.) 
The reduction modulo $q$ map can be extended to a
homomorphism $r_q$ from $\fS^{-1}\FZ[x_1,\cdots, x_n]$ to the total quotient ring 
of $\FZ/q\FZ[x_1,\cdots, x_n]$. 

We will prove $d q x_1\cdots x_n S^N(C)_B \in \fS^{-1}\FZ[x_1, \cdots, x_n].$
This will complete the proof of the theorem
since $$\fS^{-1}\FZ[x_1, \cdots, x_n] \cap q^{-(n-1)}\FZ[x_1, \cdots, x_n] 
= \FZ[x_1, \cdots, x_n].$$ 
Exploiting the decomposition $S^N(C) = S_{inner}^N(C) + S_{outer}^N(C)$, we will
show suitable multiplication of $S_{inner}^N(C)_B$ and $S_{outer}^N(C)_B$ belong to 
$\fS^{-1}\FZ[x_1, \cdots, x_n]$, respectively.

First, we claim that 
\begin{equation} \label{integrality-inner-sum}
d S_{inner}^N (C)_B \in \fS^{-1}\FZ[x_1, \cdots, x_n].
\end{equation}
By Prop.\ref{inner-sum} and the same arguments as above, 
$d S_{inner}^N (C)_E$ can be written as a quotient of two integral polynomials with denominator
$\prod_{k \neq j} (v_k^{(j)*})_E$ where the product is 
over all pairs $1 \leq j, k \leq n$ with $k \neq j$.
Changing the basis from $E$ to $B$, 
it is enough to show that $(v_k^{(j)*})_B$ is a primitive
polynomial for $k \neq j$. The coefficients of $(v_k^{(j)*})_B = 
\sum_l \langle v_k^{(j)*} , w_l \rangle x_l$ are the $k$-th row of the integral matrix
$\begin{pmatrix}\langle v_k^{(j)*} , w_l \rangle\end{pmatrix}_{k,l}$ of determinant $q$. 
Since $v_j^{(j)*} = q w_j^*$, the $j$-th row of this matrix is a multiple of $q$. 
Hence other rows must be primitive.
This last statement that $(v_k^{(j)*})_B$ is primitive for $k \neq j$ also shows that
\begin{equation} \label{integrality-outer-sum}
d(v_j^{(j)*}S^N(D_j))_B = d qx_j S^N(D_j)_B \in \fS^{-1}\FZ[x_1, \cdots, x_n]
\end{equation}
since $(v_k^{(j)*})_B$'s are the polynomials appearing 
in the denominator of $d S^N(D_j)_B$.
Therefore 
$$
d qx_1 \cdots x_n S_{outer}^N(C) =d qx_1 \cdots x_n \sum_{j=1}^{n} S^N(D_j)
\in \fS^{-1}\FZ[x_1, \cdots, x_n].
$$ 
\end{proof} 

\begin{cor}\label{integrality-s}
For $\br = (r_1, r_2,\cdots, r_n) \in \FZ^n_{>0}$,
let $N=|\br|$ and $d' = \frac{d_{N,n}}{\br !}$.
Then $d'q^{N-n+1} s_{\br}(q;p_1,\cdots,p_{n-1})$ is an integer
for any $(q;p_1,\cdots,p_{n-1}) \in I_n$. 
\end{cor}

\begin{proof}
We use induction on $n$. 
By the last theorem, $d'q^{N-n+1} t_{\br} $ is an integer.
Since $2^k d_{N-k,n-k}$ divides $d_{N,n}$ for any integer $0 \leq k \leq n$,
applying induction hypothesis to 
the second equation (multiplied by $d' q^{N-n+1}$) in Thm.\ref{relation-t-s}
completes the proof. 
\end{proof}


\section{Reduction mod $q$ of generalized Dedekind sums}

Let $C = \co(w_1,\cdots,w_n)$ be the lattice cone 
in $V = \FR^n$ generated by
$w_1 = e_1, \cdots, w_{n-1} = e_{n-1}, w_n = (p_1,\cdots, p_{n-1}, q) \in \FZ^n$
as in the paragraphs before Thm.\ref{integrality} and let $B = \{w_1,\cdots,w_n\}$.

Our next task is to consider the mod-$q$ reduction  of the integral polynomial
$d \td_C^N(x_1,\cdots,x_n) = d qx_1 \cdots x_n S^N(C)_B$, where 
$d = d_{N,n} \in \FZ$ is the constant given in Def.\ref{definition-d}.
We keep other notations of last section.
Recall that the reduction mod $q$ map can be extended
to $\fS^{-1}\FZ[x_1, \cdots, x_n]$. 
Since $d S_i^N(C)_B$ is already in this subring
(\ref{integrality-inner-sum}), $d qx_1 \cdots x_n S_i^N(C)_B$ $\pmod{q}$ vanishes 
and we have
\begin{equation}\label{red-cone}
\begin{split}
d qx_1 \cdots x_n S^N(C)_B & \equiv d q x_1 \cdots x_n S_{outer}^N (C)_B \\
&\equiv  \sum_{j=1}^n d qx_1 \cdots x_n S^N(D_j)_B \pmod q ,
\end{split}
\end{equation}
where $D_1,\cdots, D_n$ are given in the paragraph before Prop.\ref{inner-sum}.
Fix $1 \leq j \leq n$ and recall that the generators of $D_j = \co(v_1^{(j)},\cdots, v_n^{(j)})$
are ordered such that $v_i^{(j)} = w_i$ for $i \neq j$. Let $v_j^{(j)} = (b_1, \cdots, b_n)$.
It was shown $v_j^{(j)*} = q w_j^*$ in the proof of Thm.\ref{integrality}
and some calculation yields for $i \neq j$,
\begin{equation*}
 v_i^{(j)*} = 
 \left\{\begin{array}{ll}
 w_i^* + (b_n p_i - b_i q)w_j^*, &  \text{if $i \neq n$}, \\
 w_i^* - b_n w_j^*, & \text{if $i = n$}.
\end{array}\right.
\end{equation*}
We have shown that $v_1^{(j)}\wedge \cdots \wedge v_n^{(j)} 
= e_1 \wedge \cdots \wedge e_n$
and this implies $b_j q - b_n p_j = 1$ if $j \neq n$ and $b_n = 1$ if $j = n$. 
Hence we have $(v_j^{(j)*})_B = q x_j$ and for $i \neq j$,
\begin{equation*}
(v_i^{(j)*})_B \equiv \left\{ \begin{array}{lll}
x_i - p_i p_j^{-1} x_j &\mod q & \text{if  $j \neq n$ and $i \neq n$} \\
x_i + p_j^{-1} x_j &\mod q & \text{if $j \neq n$ and $i = n$} \\
x_i + p_i x_j &\mod q & \text{if  $j = n$}
\end{array}\right.
\end{equation*}
It is convenient to put $p_n = -1$. Then the above equation can be written as
\begin{equation}\label{red-face}
(v_i^{(j)*})_B \equiv x_i - p_i p_j^{-1} x_j \mod q \quad\text{for any $i \neq  j$}.
\end{equation}
Note the reduction mod $q$ of $(v_i^{(j)*})_B$ for $i \neq j$ does not depend on 
the lattice vector $v_j^{(j)}$. 
Let us write $t_\br(C)$ for $t_{r_1,\cdots,r_n}(q;p_1,p_2,\cdots,p_{n-1})$
in Def.\ref{definition-t}. Hence we have
\begin{equation*}
\td_C(x_1,\ldots,x_n) =\sum_{N=0}^{\infty} qx_1 \cdots x_n S^N(C)_B
=  \sum^\infty_{N=0} \sum_{|\br| =N} (-1)^{N} q^{N-n+1}
\frac{t_{\br}(C) }{\br !} \bx^{\br} .
\end{equation*}

\begin{thm}\label{thm-congruence}
For $(q;p_1,\cdots,p_{n-1}) \in I_n$ and 
$\br = (r_1, r_2,\cdots, r_n) \in \FZ_{\geq 1}^n$,
let $N= |\br|$, $d = d_{N,n}$ and $C = C(q;p_1,\cdots,p_{n-1})$.
Then $d  q^{N-n+1} t_{\br}(C)/\br !$ is an integer and 
\begin{equation}\label{eq-thm}
\frac{d q^{N-n+1}}{\br !} t_{\br}(C) \; \equiv \;  \sum_{\bm} (-d)
\frac{B_{\bm}}{\bm!}\prod_{i=1}^{n}\binom{m_i -1}{r_i-1} p_i^{m_i - r_i}
\mod q .
\end{equation}
The summation is over the set of $n$-tuples 
$\bm = (m_1,\cdots, m_n) \in \FZ_{\geq 0}^n$ with $|\bm| = N$ 
such that at least one of its coordinates is zero. (Recall we have put $p_n = -1$.)
\end{thm}

\begin{remark}\label{summation}
Because of vanishing of some Bernoulli numbers 
and binomial coefficients, we can restrict the summation in the above statement 
over $\bm$ such that $m_i$ is even if $m_i > 1$ and $m_i \geq r_i$ if $m_i \neq 0$.
\end{remark}

\begin{remark}
$ \frac{dB_{\bm}}{\bm!}$ in \eqref{eq-thm} are integers by the definition of $d$.
\end{remark}

\begin{proof}
The first statement is proved in Thm.\ref{integrality}.
It remains to prove \eqref{eq-thm}.
Let  $\br = (r_1, r_2,\cdots, r_n) \in \FZ_{\geq 1}^n$ and let $N = |\br| = \sum r_i$.
Let $M^N$ be the set of $\bm \in \FZ_{\geq 0}^n$ with $|\bm| = N$.

The left hand side of \eqref{eq-thm} is ($(-1)^N$-times) the coefficient of
$\bx^\br=x_1^{r_1}\cdots x_n^{r_n}$ in the power series expansion of
$$
d \td_C^N(x_1,\ldots,x_n) = dq x_1\cdots x_n S^N(C)_B.
$$ 
It can be computed as the \textit{iterated constant term} of 
$$d qx_1 \cdots x_n S^N(C)_B/(x_1^{r_1} \cdots x_n^{r_n})$$ 
with respect to $\fA=\{x_1,\ldots,x_n\}$ (c.f. App.B).
In general, iterated constant term depends on the order of the hyperplanes in the flag. 
As $dq x^1\cdots x^n S^N(C)/(x_1^{r_1} \cdots x_n^{r_n})$ is 
a Laurent polynomial in $x_1,\cdots, x_n$, 
the iterated constant term is simply the constant term
and the order in $\fA$ does not matter. We calculate
the iterated constant terms of rational functions (the  Todd polynomials of
cones appearing in the nonsingular decomposition) which are not Laurent polynomials in $x_1,\cdots,x_n$, 
and the order of hyperplanes should be fixed once. 
However,  before summing up, individual iterated constant
term is dependent on the order.

We know that the homogeneous rational
function $d qx_1 \cdots x_n S^N(D_j)_B$ belongs to
$\fS^{-1}\FZ[x_1, \cdots, x_n]$ by \eqref{integrality-outer-sum}. 
From \eqref{red-face} and \eqref{todd-N}, we see that for $1 \leq j \leq n$
\begin{multline*}
d qx_1 x_2 \cdots x_n S^N(D_j)_B \equiv  (-1)^N qx_1 x_2 \cdots x_n \; \times \\
\sum_{\bm \in M^N} 
d \frac{B_{\bm}}{\bm !} \;
\prod_{\substack{1\leq i \leq n\\i\ne j}}(x_i - p_i p_j^{-1}x_j)^{m_i-1} \cdot (qx_j)^{m_j-1} \pmod{ q} .
\end{multline*}
We will apply this to (\ref{red-cone}).
First, note that the term corresponding 
to $\bm=(m_1,\cdots, m_n)$ in the above summation
belongs to $\fS^{-1}\FZ[x_1, \cdots, x_n]$ if $m_j \geq 1$.
When multiplied by $qx_1\cdots x_n$, 
it vanishes as a rational function with coefficient in $\FZ/q\FZ$. 
Hence the above remains the same if we 
restrict the summation over terms with $m_j = 0$:
\begin{multline}\label{red-j-cone}
d qx_1 x_2\cdots x_n S^N(D_j)_B \equiv  (-1)^N x_1 
\cdots\widehat{x_j}\cdots x_n\;\times \\
\sum_{\substack{\bm \in M^N\\ m_j=0}} d \frac{B_{\bm}}{\bm !}\;
\prod_{\substack{1\leq i \leq n \\ i\ne j}}(x_i - p_i p_j^{-1}x_j)^{m_i-1} \pmod{q} 
\end{multline}

So in view of \eqref{red-cone} 
we have to calculate the iterated constant term with respect to $\fA$
of the right hand side of the following congruence equation mod $q$:

\begin{equation}\label{eq-mod-q}
\frac{d qx_1 \cdots x_n S^N(C)_B}{x_1^{r_1} \cdots x_n^{r_n}}
\equiv \sum_{j=1}^n (-1)^N
\sum_{\substack{\bm \in M^N \\m_j=0}} d \frac{B_{\bm}}{\bm !}  \;
\prod_{\substack{1\leq i \leq n \\ i\ne j}}
\frac{(x_i - p_i p_j^{-1}x_j)^{m_i-1}}{x_i^{r_i -1}} \cdot \frac{1}{x_j^{r_j}} \;  .
\end{equation}

Let $F_j$ denote the $j$-th rational function in the summation of
the right hand side of (\ref{eq-mod-q}):
$$
F_j := \sum_{\substack{\bm \in M^N \\m_j=0}} d \frac{B_{\bm}}{\bm !}  \;
\prod_{\substack{1\leq i \leq n \\ i\ne j}}
\frac{(x_i - p_i p_j^{-1}x_j)^{m_i-1}}{x_i^{r_i -1}} \cdot \frac{1}{x_j^{r_j}} \; .
$$
We compute the iterated constant term of $F_j$.
%
For each $\bm \in M^N$ with $m_j = 0$, we have 
\begin{multline}\label{constant-term}
\ct_{x_{j+1}}\circ\cdots\circ\ct_{x_n}
\prod_{\substack{1\leq i \leq n \\ i\ne j}}
\frac{(x_i - p_i p_j^{-1}x_j)^{m_i-1}}{x_i^{r_i -1}}\cdot\frac{1}{x_j^{r_j}} = \\ 
\prod_{i=1}^{j-1}\frac{(x_i - p_i p_j^{-1}x_j)^{m_i-1}}{x_i^{r_i-1}} \cdot
\prod_{i=j+1}^{n}\binom{m_i -1}{r_i-1} (-p_i p_j^{-1})^{m_i - r_i} \cdot
x_j^{\sum_{i \geq j} m_i - r_i} .
\end{multline}
Here  individual constant term is computed 
by obtaining the Laurent series of $(x_i-p_ip_j^{-1}x_j)^{m_j-1}$ w.r.t.
$x_i$ through binomial expansion. 
Now the constant term with respect to $x_j$ of the above depends on the sign of 
$e = \sum_{i \geq j} m_i - r_i$. 
If $e  > 0$, then $\ct_{x_j}$ of (\ref{constant-term}) vanishes. 
If $e \leq 0$, then $\ct_{x_j}$ of (\ref{constant-term}) is equal to
\begin{multline}\label{j-constant-term}
\ct_{x_j}\left(
\prod_{i=1}^{j-1}\frac{(x_i - p_i p_j^{-1}x_j)^{m_i-1}}{x_i^{r_i-1}}  \cdot
\prod_{i=j+1}^{n}\binom{m_i -1}{r_i-1} (-p_i p_j^{-1})^{m_i - r_i} x_j^{\sum_{i \geq j} m_i - r_i}\right)\\= 
\sum_{a_1, \cdots, a_{j-1}} 
\prod_{i=1}^{j-1}\binom{m_i -1}{a_i} (-p_i p_j^{-1})^{a_i}x_i^{m_i - r_i-a_i}\cdot
\prod_{i=j+1}^{n}\binom{m_i -1}{r_i-1} (-p_i p_j^{-1})^{m_i - r_i} ,
\end{multline}
where the summation is over non-negative integers $a_1, \cdots , a_{j-1}$ with
$a_1 + \cdots + a_{j-1} = -e = \sum_{i \geq j} r_i - m_i$.
Now it is direct to see that 
$\ct_{x_1}\circ\ct_{x_2}\circ\cdots\circ\ct_{x_{j-1}}$ of \eqref{j-constant-term}
is supported at  $a_i = m_i - r_i$ for $1 \leq i \leq j-1$.
This can be satisfied only if $m_i \geq r_i$ for $1 \leq i \leq j-1$. 
\textit{A priori} unless $m_i>0$ for $1\le i \le j-1$, the iterated constant term vanishes. We will be in need of this later.
Hence  if  $\bm \in M^N$ is such that 
(a) $\sum_{i \geq j} m_i - r_i \leq 0$ and
(b) $m_i \geq r_i$ for $1 \leq i \leq j-1$, then we have
\begin{multline*}
\ict_{x_1,x_2,\cdots,x_n}
\prod_{\substack{1\leq i \leq n \\ i\ne j}}
\frac{(x_i - p_i p_j^{-1}x_j)^{m_i-1}}{x_i^{r_i -1}}\cdot\frac{1}{x_j^{r_j}} 
=\prod_{\substack{1\leq i \leq n \\ i\ne j}}\binom{m_i -1}{r_i-1} (-p_i p_j^{-1})^{m_i - r_i} \\
= \prod_{\substack{1\leq i \leq n \\ i\ne j}}\binom{m_i -1}{r_i-1} p_i^{m_i - r_i}
\cdot(-p_j^{-1})^{\sum_{i\neq j}m_i - r_i}
=(-1)\prod_{i=1}^{n}\binom{m_i -1}{r_i-1} p_i^{m_i - r_i}\\
\end{multline*}
where the last equality comes from $\sum_{i \neq j} m_i - r_i = r_j - m_j$ and $m_j = 0$.
Note (b) implies (a)  since $\sum_{i=1}^n m_i - r_i = 0$.
Taking the summation of above iterated constant terms,
we have
\begin{equation*}
\ict_{x_1,x_2,\cdots,x_n}(F_j) = 
\sum_{\bm \in M_j} (-d)
\frac{B_{\bm}}{\bm!} \prod_{i=1}^{n}\binom{m_i -1}{r_i-1} p_i^{m_i - r_i}
\end{equation*}
where the summation is over the set $M_j$ of $n$-tuples $\bm = (m_1,\cdots,m_n)$
of non-negative integers such that $|\bm| = N, \; m_j = 0$ 
and $m_i \geq r_i$ for $1 \leq i \leq j-1$. 

Notice that the sets 
$M_1, \cdots, M_n$ are disjoint to each other 
since $\bm=(m_1,\ldots,m_n)$ belongs to $M_j$ iff
$m_j=0$ and $m_i >0$ for $i<j$.
Since $\cup_{j=1}^n M_j$ is seen to be the set over which the summation 
of (\ref{eq-thm}) is taken,
we see that the righthand side of \eqref{eq-thm} is given as
$$
\sum_{j = 1}^n \ict_{x_1,x_2,\cdots,x_n}(F_j) = \sum_\bm(-d)
\frac{B_\bm}{\bm!}\prod_{i=1}^n \binom{m_i -1}{r_i-1} p_i^{m_i - r_i}.
$$
We have calculated the iterated constant terms of rational functions $F_j$ with 
coefficients in $\FQ$. But these calculations are still valid if we view $F_j$ as
rational functions with coefficients in $\FZ/q\FZ$. They are $\fA$-admissible
in the sense of App.~B. This implies that iterated constant terms
with respect to $\fA$ of both sides of (\ref{eq-mod-q}) are congruents modulo $q$,
which completes the proof.
\end{proof}

\begin{thm}\label{congruence-s}
For $\br = (r_1, r_2,\cdots, r_n) \in \FZ_{\geq 1}^n$,
let $N= |\br|$ and $d = d_{N,n}$. 
Then for $(q;p_1,\cdots,p_{n-1}) \in I_n$ we have
$$\frac{dq^{N-n+1}}{\br!} s_{\br }(q;p_1,\cdots,p_{n-1})  \in \FZ$$ 
and 
\begin{equation}\label{eq-congruence-s}
\frac{dq^{N-n+1}}{\br!} s_{\br }(q;p_1,\cdots,p_{n-1}) \; \equiv \;  
\sum_{\bm } (- d )\frac{B_{\bm}}{\bm!}\prod_{i=1}^{n}\binom{m_i -1}{r_i-1} 
p_i^{m_i - r_i} \mod q .
\end{equation}
The summation is over the set of $n$-tuples 
$\bm = (m_1,\cdots, m_n)$ of non-negative \textbf{even} integers with $|\bm| = N$ 
such that at least one of its coordinates is zero. (Recall we have put $p_n = -1$.)
\end{thm}
\begin{remark}
As in Remark \ref{summation},  
we can restrict the summation in the above statement 
over $\bm$ such that $m_i \geq r_i$ if $m_i \neq 0$
\end{remark}
\begin{proof}
The first statement is proven in Cor.\ref{integrality-s}.
If $N$ is odd, both sides of (\ref{eq-congruence-s}) 
vanish by Prop.\ref{odd-degree-sum}.
Hence we assume $N$ is even. As in the proof of the same proposition,
we will use induction on $n$. 
Let $d' = d_{N,n}/\br !$.

When $n=1$ (so $N = r_1$), we have $s_{r_1}(q) = B_{r_1}$ by definition
(Remark \ref{defn-exceptional}). In this case, (\ref{eq-congruence-s})  is
$d_{N,1}q^N B_N/N! \equiv 0 \mod q$, which is obvious since
$d_{N,1} B_N/N!$ is an integer by definition of $d_{N,1}$.
The case when $n=2$ is Thm.1.1 in \cite{J-L2}.

Let $M$ be the set of $\bm = (m_1,\cdots, m_n) \in \FZ_{\geq 0}^n$ 
with $|\bm| = N$ such that at least one of its coordinates is zero and
satisfies the conditions given in Remark \ref{summation}.
In particular, if the $i$-th coordinated of  an $\bm \in M$ is odd,
then $r_i =1$ and $m_i = 1$.

Let $J = \{ \;1 \leq  j \leq n \; | \; r_j = 1 \; \}$. For $T \subset J$ with $|T|$ even,
let $M(T)$ be the set of $(m_1,\cdots, m_n) \in M$ such that 
$m_j = 1$ for $j \in T$ and $m_j \neq  1$ (hence is even) if $j \not\in T$. 
Note that the summation in (\ref{eq-congruence-s}) is over $M(\emptyset)$
while the summation in (\ref{eq-thm}) is over $M$.

If $|J| \leq  1$, then $s_\br = t_\br$ 
by Thm.\ref{relation-t-s}
and $M(\emptyset) = M$. Hence the above corollary is a restatement of 
the last theorem. In general, let $T$ be a nonempty subset of $J$ of even order $k$.
Then we have by induction assumption that 
the partial sum over $M(T)$ of the sum  in (\ref{eq-thm}) is congruent o 
$B_1^k = (-1/2)^k$ times 
$d'q^{N-n+1} s_{r_{j_1},\cdots, r_{j_{n-k}}} (q; p^T_{j_1},\ldots, p^T_{j_{n-k-1}})$
modulo $q$.
in the notation of Thm.\ref{relation-t-s}
(so that $\{j_1,\cdots,j_{n-k}\} = \{1,\cdots, n\} \backslash T$). Since $M$ is the disjoint
union of $M(\emptyset)$ and $M(T)$'s for $T$ ranging over nonempty subsets of $J$
of even order, we see from Thm.\ref{thm-congruence} and the second equation
(multiplied by $d' q^{N-n+1}$) in Thm.\ref{relation-t-s} that 
$d'q^{N-n+1} s_{\br}(q;p_1,\cdots,p_{n-1})$ is congruent mod $q$ to
the partial sum over $M(\emptyset)$, which proves (\ref{eq-congruence-s}).
\end{proof}


\section{Equidistribution of generalized Dedekind sums and exponential sums}

Given $\br=(r_1,\cdots, r_n) \in \FZ^n_{> 0}$, let $N = |\br|$.
Let $f_{\br}(p_1,\cdots,p_{n-1})$ be the right hand side of (\ref{eq-congruence-s})
considered as a Laurent polynomial in $p_1,\cdots, p_{n-1}$ with integral coefficients.
For $(q;p_1,\cdots,p_{n-1}) \in I_n$,
we have by Thm.\ref{congruence-s}
\begin{equation}\label{fractional-part-s}
\left< \frac{d_{N,n}q^{N-n}}{\br!} s_{\br}(q;p_1,\cdots,p_{n-1}) \right>
= \left<\frac{1}{q}f_{\br}(p_1,\cdots,p_{n-1}) \right>
\end{equation}
where $\left<t\right> = t - [t]\in[0,1)$ denotes the fractional part of $t$ and in the right hand side
we take $p_i^{-1}$ to be an inverse modulo $q$(i.e. any integer
such that $p_i^{-1}p_i \equiv 1 \mod q$).

The goal of this section is to show the equidistribution 
of this sequence of numbers for varying $(q;p_1,\cdots,p_{n-1})$ with fixed $\br$
following the line of \cite{J-L2}. Since these sums are multi-indexed, the classical definition
of equidistribution of sequence of numbers in $[0,1)$ can not be applied directly.
Instead, we take a variant of Weyl's equidistribution criterion as our definition.
For $x \in \FR_{>0}$, let
$I_n(x)$ be the set of $(q;p_1,\cdots,p_{n-1}) \in I_n$ with $q < x$ (For definition of $I_n$, see \textsc{Notations}).
Again the limit of the average of the point mass weakly converges to the probability measure on $[0,1)$ which is the restriction of the standard Lebesgue measure.
\begin{defn}[Weyl's equidistribution]\label{definition-equidistribution}
Let $A$ be a set of numbers $a_{(q;p_1,\cdots,p_{n-1})}$ indexed by $(q;p_1,\cdots,p_{n-1})\in I$.
We say the set $A$ is equidistributed in $[0,1)$ if for any nonzero integer $k$, we have
\begin{equation*}
\lim_{x \rightarrow \infty} \frac{1}{|I_n(x)|} 
\sum_{(q;p_1,\cdots,p_{n-1}) \in I_n(x)} 
\exp\left(2\pi i k\cdot a_{(q;p_1,\cdots,p_{n-1})}\right) = 0 .
\end{equation*}
\end{defn}
More generally, given any Laurent polynomial $f(x_1,\cdots,x_{n-1})$
with integral coefficients, we may consider the fractional parts of 
$f(p_1,\cdots,p_{n-1})/q$ as in the first paragraph of this section. 
We will prove that this set of numbers in $[0,1)$ is equidistributed in the above sense
if $f$ satisfies the condition (H) given below. For each $i$, $f$ can be written as
$f = \sum_j g_j(x_1,\cdots, \widehat{x_i},\cdots, x_{n-1}) x_i^j$, a Laurent polynomial
in $x_i$ with coefficients in the ring of Laurent polynomials in other variables.

In the remaining section, we suppose that $f$ satisfies the following:


\begin{verse}[H]
There exists $i$ such that when written 
as a Laurent polynomial in $x_i$
as above, the coefficient of the highest degree (in $x_i$) term is a monomial
in the other variables $x_1,\cdots, \widehat{x_i},\cdots, x_{n-1}$.
\end{verse}

Note that the Laurent polynomial $f_{\br}$ of (\ref{eq-congruence-s}),
which gives the fractional part of generalized Dedekind sums $s_{\br}$
satisfies this assumption: for each variable $x_i$, the coefficient of
the highest degree term $x_i^{N-r_i}$ is a monomial.

\begin{prop}\label{equidistribution}
Let $f(x_1,\cdots,x_{n-1}) \in \FZ[x_1^{\pm},x_2^{\pm},\cdots,x_{n-1}^{\pm}]$
satisfying the assumption (H).
Then the fractional parts of $f(p_1,\cdots,p_{n-1})/q$ are equidistributed 
in the sense of Def.\ref{definition-equidistribution} : for any nonzero integer $k$,
\begin{equation*}
\lim_{x \rightarrow \infty} \frac{1}{|I_n(x)|} 
\sum_{(q;p_1,\cdots,p_{n-1}) \in I_n(x)} 
\exp\left(\frac{2\pi i k}{q}f(p_1,\cdots,p_{n-1})\right) = 0 .
\end{equation*}
\end{prop}

\begin{thm}\label{thm-equidist}
Let $\br=(r_1,\cdots, r_n) \in \FZ^n_{> 0}$ and
suppose $|\br|$ is even.
Then fractional parts of generalized Dedekind sums given as 
$$\frac{d_{N,n}q^{N-n}}{\br!} s_{\br}(q;p_1,\cdots,p_{n-1})$$
are equidistributed in $[0,1)$.
\end{thm}

The proof consists of three steps. At each step, we estimate the following
exponential sums for $q$ a prime, a prime power and any composite number, respectively.
Then this together with the estimation of the order of $I_n(x)$ completes the proof.

For a positive integer $q$,  let $K(f,q)$ be the following exponential sum,
which is a partial sum of the above over $n$-tuples with given $q$.
(Replacing $f$ by $kf$, we may consider only the case when $k=1$ in the theorem.)
\begin{equation}\label{definition-exp-sum}
K(f,q)= \sum_{p_1,\cdots,p_{n-1}} \be_q(f(p_1,\cdots,p_{n-1})) ,
\end{equation}
where the summations are  over $(p_1,\cdots,p_{n-1}) \in \FZ^{n-1}$ with $1 \leq p_j < q$
relatively prime to $q$  $(1 \leq j \leq n-1)$ and $\be_q(x):=\exp(2\pi i \frac{x}{q}).$

\begin{prop}\label{prop:6.4}
There exists a constant $C_1$ depending only on $f$ such that 
we have for almost all prime $q$ (hence for any prime $q$ if we enlarge $C_1$), 
\begin{equation*}
\left| K(f,q) \right| \leq C_1 q^{(n-1) -\frac{1}{2}} .
\end{equation*}
\end{prop}

\begin{proof}Note the trivial counting gives the estimate $\leq q^{n-1}$. 
When $f$ is a polynomial, this is Prop.3.8 in \cite{Del} with $C_1 = \deg(f) -1$.
But the same proof can be applied to Laurent polynomials with slight modification. 
Really, from (3.5.2) in \textit{loc.cit.} with $X_0 = \FP^1$, we obtain
the desired estimate when $f$ is a Laurent polynomial in one variable:
we can take $C_1 = \nu_0(f) + \nu_{\infty}(f)$ where $\nu_z(f)$ denotes
the order of pole at $z$ and we put $\nu_z(f) = 0$ if $f$ is regular at $z$.
(Let us call this integer $C_1 = C_1(f)$ the \textit{width} of the 
one variable Laurent polynomial $f$.)

In general, suppose $n-1 \geq 2$ and $f$ satisfies (H) with $i=1$. Then
for generic $q$, the reduction mod $q$ of $f(x_1,p_2,\cdots, p_{n-1})$ is
a nonconstant Laurent polynomial in $x_1$ for any 
$p_2,\cdots, p_{n-1} \in (\FZ/q\FZ)^*$. Clearly, the exponential sum
of this one variable Laurent polynomial
can be bounded by $\sqrt{q}$ times a constant $C_1$ which depends only on $f$.
Hence we have
\begin{equation*}
\left|K(f,q)\right| 
\leq  \sum_{p_2,\cdots,p_{n-1}}\left|\sum_{p_1}\be_q(f(p_1,\cdots,p_{n-1}))\right|
\leq C_1 q^{(n-2)}q^{\frac{1}{2}}.
\end{equation*}
\end{proof}

Next, we consider the case when $q$ is a power of a prime $p$. 

\begin{prop}\label{prime power}
There exist a constant $C_2$ and integers $d > 0, D$ depending only on $f$ 
such that for any prime power $q$ relatively prime to $D$, we have
\begin{equation*}
\left|K(f,q)\right|\leq C_2 q^{(n-1) - \frac{1}{3d}}.
\end{equation*}
\end{prop}

\begin{proof}
Let $q = p^{\alpha}$ with $p$ a prime. We assume $\alpha \geq 2$ 
since the other case is treated in the last proposition.
As in the proof of it, we assume $f$ satisfies (H) with $i=1$.
First, suppose $\alpha = 2\beta$ is even. 
Then by applying Lemma 12.2 of \cite{IK} to the one variable
Laurent polynomial $f(x_1,p_2,\cdots,p_{n-1})$, we have
\begin{equation}\label{even-power}
K(f,p^{2\beta}) = \sum_{p_2,\cdots,p_{n-1}} 
p^{\beta}\sum_{p_1} \be_{p^{2\beta}}(f(p_1,p_2,\cdots,p_{n-1})) ,
\end{equation}
where the first summation is over $p_2,\cdots, p_{n-1} \in (\FZ/q\FZ)^*$
and the second summation is over the set  
$A_1$ of $p_1 \in (\FZ/p^{\beta}\FZ)^*$ with
$\partial_1 f(p_1,p_2,\cdots, p_{n-1}) \equiv 0 \mod p^{\beta}$. 
($\partial_1$ denotes the partial derivative w.r.t. the variable $x_1$.)
Let $d$ be the width with respect to $x_1$ of the Laurent polynomial 
$\partial_1 f(x_1,x_2,\cdots, x_{n-1})$.
(The width of a Laurent polynomial was defined in the proof of last proposition.)
By Cor.\ref{cor-app-cong} in Appendix \ref{app-cong}, we have
\begin{equation*}
|K(f,p^{2\beta})| \leq q^{n-2} p^{\beta} |A_1|  
\leq  d q^{(n-2)} p^{2\beta -\frac{\beta}{d}} = d q^{(n-1) - \frac{1}{2d}} \; .
\end{equation*}
To apply the corollary, the coefficient $D$ of the highest degree (with respect to $x_1$) term
in $\partial_1 f(x_1,x_2,\cdots, x_{n-1})$ should be relatively prime to $p$.
This excludes a finite number of primes dividing $D$.

The case when  $q= p^{2\beta+1}$ with $\beta \geq 1$ can be treated similarly.
We have by Lemma 12.3 of \textit{loc.cit.} (with $\bp$ denoting $(p_1,p_2,\cdots, p_{n-1})$),
\begin{equation}\label{odd-power}
K(f,p^{2\beta+1}) = 
\sum_{p_2,\cdots, p_{n-1}}p^\beta \sum_{p_1}\be_{p^{2\beta+1}}
\left(f(\bp)\right)G_p(\bp),
\end{equation}
where the first summation is over $p_2,\cdots, p_{n-1} \in (\FZ/q\FZ)^*$
and the second summation is over the same subset $A_1$ of $(\FZ/p^{\beta}\FZ)^*$ as above. 
And we have put
\begin{equation*}
G_p(\bp)= \sum_{y\in \FZ/p\FZ}\be_p\left(d(\bp)y^2+h(\bp)p^{-\beta}y\right)
\end{equation*}
with $d(\bp) =\partial_1^2 f(\bp)/2$ and  $h(\bp)=\partial_1 f(\bp)$.
Since $ |G_p(x)|\leq p$, we obtain 
\begin{equation*}
|K(f,p^{2\beta+1})| \leq q^{n-2} p^{\beta + 1} |A_1|  
\leq  d q^{(n-2)} p^{2\beta + 1  -\frac{\beta}{d}} \leq d q^{(n-1) - \frac{1}{3d}} \; 
\end{equation*}
since $3\beta \geq \alpha = 2\beta + 1$.
This completes the proof with $C_2 = d$ when $\alpha \geq 2$.
\end{proof}

Let us consider the case when  $q$ has several prime factors.
We have the following effect of the Chinese remainder theorem for the exponential sums.
\begin{lem}\label{pro2} 
Let $f$ be an integer coefficient Laurent polynomial 
and $q_1, q_2>1$ be relatively prime integers. Then we have 
$$
K(f,q_1q_2)=K(f,q_1)K(f,q_2).
$$
\end{lem}

\begin{proof}
This is an easy consequence of Fubini theorem.
\end{proof}

From the lemma, the following proposition follows immediately. 

\begin{prop}\label{bound}
Let $C_2, d$ and $D$ be as in Prop.\ref{prime power}.
Then for any integer $q > 1$ relatively prime to $D$
\begin{equation*}
\left|K(f,q)\right| \leq \left(C_2\right)^{\omega(q)} q^{(n-1)-\frac{1}{3d}},
\end{equation*}
where $\omega(q)$ is the number of prime factors of $q$.
\end{prop}

Note that $\omega(q)$   has a well-known estimate
\begin{equation*}
\omega(q)\sim \log\log q.
\end{equation*}
For sufficiently large $q$, we have that
$$
C_2^{\omega(q)}\leq  C_2^{c \log \log q}\leq (\log q)^{c \log C_2}.
$$
Thus, we obtain that  for any $\epsilon>0$,
$$
C_2^{\omega(q)} \ll q^{\epsilon}.
$$
Therefore, we have the following bound:

\begin{prop}\label{final estimate}
Let $d$ and $D$ be as in Prop.\ref{prime power}. 
Then for any $\epsilon > 0$ and any integer $q > 1$ relatively prime to $D$, 
we have
$$
\left|K(f, q)\right|\ll q^{(n-1) -\frac{1}{3d} + \epsilon}
$$
\end{prop}

For $x>1$ let $\phi(x):=|(\FZ/[x]\FZ)^*|$ be the Euler's phi function. 
\begin{prop}
For any $\epsilon > 0$, we have
$$|I_n(x)|=\sum_{q<x} \phi(q)^{n-1}\gg x^{n -\epsilon}.$$ 
\end{prop}

\begin{proof}
It is known that for all but finitely many positive integers $q$, 
$$\phi(q)\geq \frac{q}{e^{\gamma} \log\log q}.$$
Since for any $\epsilon > 0 $, there is a positive number $C_{\epsilon}$ such that
$$\log\log q \leq C_{\epsilon} q^{\epsilon},$$ 
we have that
\begin{equation*}
\sum_{q<x} \phi(q)^{n-1} \gg \sum_{q<x} \left(\frac{q}{e^{\gamma} \log\log q}\right)^{n-1}
\gg \sum_{q<x} q^{(1-\epsilon)(n-1)}\gg x^{n-\epsilon} .
\end{equation*}
\end{proof}

Now we come to the proof of Prop.\ref{equidistribution}. 
This will be done by combining previous estimates.
\begin{proof}[Proof of Prop.\ref{equidistribution} ]
To estimate $\sum_{0<q<x} |K(f, q)|$, we need to extend the result
of Prop.\ref{final estimate} to arbitrary integer $q >1$. 
For $D, d$ given as in Prop.\ref{prime power},
we define a multiplicative arithmetic function $\chi_D$ by
$$
\chi_D(q) := \prod_{p | D} p^{\ord_p q}
$$
where the product is over the set of primes
dividing $D$. Since we have a trivial estimate $|K(f,q)| \leq q^{n-1}$, if we multiply
the right hand side of inequalities in Prop.\ref{prime power} and Prop.\ref{bound}-\ref{final estimate} by 
$\chi_D(q)^{1/3d}$, then the inequalities hold for any $q > 1$.
By (1.79) in \cite{IK}, for sufficiently large $x$,  we have 
\begin{equation*}
\sum_{q<x} \chi_D(q)^{1/3d} \leq x \prod_{p | D} \left(1-p^{-1+\frac{1}{3d}}\right)^{-1}.
\end{equation*}
By the partial summation, we have
\begin{equation*}
\sum_{q < x} q^{(n-1) -\frac{1}{3d} + \epsilon}\cdot\chi_D(q)^{1/3d} 
\ll x^{n -\frac{1}{3d} + \epsilon} \; .
\end{equation*}
From this estimate and the last two propositions, we have
\begin{equation*}
\frac{1}{|I_n(x)| } \sum_{0<q<x} |K(f, q)| \rightarrow 0,
\end{equation*}
as $x\rightarrow \infty$.
This completes the proof of Prop.\ref{equidistribution}.
\end{proof}

\section{Examples}
In the following, we present  the Laurent polynomials associated to some cases of generalized Dedekind sums for small
indices(thus including small dimension). Note $n=2$ case is throughly studied  in \cite{J-L2}. 
The cases considered here are  generalized Dedekind sums in 3-dimension(i.e. $n=3$)  and Dedekind-Zagier sums (i.e. 
$\br=(1,1,\ldots,1)$) in \cite{Zagier1}.


\subsection{Three dimensional Dedekind sums}
Let $n = 3$ and $N$ be even.
\noindent\textit{Example.} Let $(r_1,r_2,r_3) = (6,4,2)$. 
We have $d_{12,3} = 2^{12}\cdot3^6\cdot 5^3\cdot 7^2\cdot 11\cdot13$. Let
\begin{gather*}
A_{1,2} = 15202p_1^{-6}p_2^{-4}, \quad A_{2,3} = 638484p_1^{6}p_2^{-4}, \quad 
A_{1,3}= 228030 p_1^{-6}p_2^{8} \\
\begin{align*}
A_{1} &= 382200 p_1^{-6}p_2^{6} + 315315p_1^{-6}p_2^{4} +143000p_1^{-6}p_2^{2} + 21021p_1^{-6}\\
A_{2} &= 573300 p_1^{4}p_2^{-4} + 189189p_1^{2}p_2^{-4}+14300p_2^{-4} \\
A_{3} &= 63063 p_1^{2} + 28600 p_2^{2} .
\end{align*}
\end{gather*}
Then $f_{\br}$ is the sum of all the Laurent polynomials above. The Laurent polynomials
supported on faces of $\Delta_{\infty}(f_{\br})$ are (minus of)
$A_{1,2}. A_{2,3}, A_{1,3}, A_{1,2}+A_{2}+A_{2,3}, A_{1,3}+A_{1}+A_{1,2}$
and $A_{2,3}+A_{3}+A_{1,3}$.

\subsection {Dedekind-Zagier sums}
The Dedekind-Zagier sum $d(q;p_1,\cdots,p_{n-1})$ of \eqref{tri}
is related to the generalized Dedekind sum 
$s_{1,\cdots, 1}(q;p_1,\cdots,p_{n-1})$ \cite{Zagier1,Bernt}:
\begin{equation}\label{relation-s-d}
s_{1,\cdots, 1}(q;p_1,\ldots, p_{n-1}) 
= \frac{(-1)^{\frac{n}2 + 1}}{2^n q} d(q;p_1,\ldots,p_{n-1}).
\end{equation}
Recall that this sum is related to the coefficient $t_{1,\cdots,1}(q;p_1,\cdots, p_{n-1})$ 
of $x_1\cdots x_n$ in the Todd series of $C$ by Thm.\ref{relation-t-s}.
Since $(r_1,\cdots, r_n) = (1,\cdots,1)$ fixed here, 
let us drop it from the
notation and simply write $s(q;p_1,\cdots,p_{n-1})$ for $s_{1,\cdots, 1}(q;p_1,\cdots,p_{n-1})$
and $t(q;p_1,\cdots,p_{n-1})$ for  $t_{1,\cdots, 1}(q;p_1,\cdots,p_{n-1})$.
From (\ref{relation-s-d}) and Thm.\ref{congruence-s}, 
we deduce immediately the following.

\begin{prop}
Let $n \geq 2$ be even and let $d = d_{n,n}$. Then
$$
\frac{d}{2^n} d (q;p_1,\cdots,p_{n-1})\in \FZ.
$$ 
Moreover, 
\begin{equation}\label{Zagier-sum}
\frac{d}{2^n} d(q;p_1,\cdots, p_{n-1}) \equiv (-1)^{\frac{n}{2} +1} \sum_{\bm}
d \frac{B_{\bm}}{\bm !}\;\prod_{i=1}^{n-1}p_i^{m_i-1} \pmod{q}
\end{equation}
where the summation is over the set of $n$-tuples 
$\bm = (m_1,\cdots, m_n)$ of non-negative even integers
with $\sum_{i=1}^n m_i = n$.
\end{prop}

\noindent\textbf{Remark.} 
Let $\td_{\ev}^N(x_1,\cdots,x_k)$ be the totally even part of $N$-th Todd polynomial 
in $k$ variables (i.e. sum of terms which are of even degree in each variable). 
Then the above equation can be written as
\begin{equation*}
\frac{d}{2^n} d (q;p_1,\cdots,p_{n-1}) \equiv
 (-1)^{\frac{n}{2} +1}d 
 \frac{\td_{\ev}^n(p_1,\cdots, p_{n-1},1)}{p_1 \cdots p_{n-1}} \pmod{q}.
\end{equation*}

Note that the result on the denominator of $d(q;p_1,\cdots,p_{n-1})$ here is not sharp. 
A more precise result is 
given
in \cite{Zagier1}:
$d_n = d_{n,n}$ for small $n$ can be: 
$$d_2=2^2 \cdot 3, d_4 = 2^4 \cdot 3^2 \cdot 5 ,
d_6 = 2^6 \cdot 3^3 \cdot 5 \cdot 7 ,d_8 = 2^8 \cdot 3^4\cdot 5^2 \cdot 7 ,
d_{10} = 2^{10} \cdot 3^5 \cdot 5^2 \cdot 7 \cdot 11$$

\noindent\textit{Example.} 
For $n = 4$, $d_{n,n} = 720$ and we have
\begin{equation*}
p_1 p_2 p_3 \cdot f(p_1,p_2,p_3) =  p_1^4 + p_2^4 + p_3^4 
- 5 p_1^2p_2^2 - 5 p_2^2 p_3^2 - 5p_3^2 p_1^2 - 5 p_1^2 - 5p_2^2 - 5p_3^2 + 1.
\end{equation*}

\appendix
\section{Number of congruence solutions modulo a prime power}\label{app-cong}
In this appendix, we prove a simple estimate of the number of solutions
of a polynomial congruence equation modulo a prime power.

\begin{prop}
Let $f \in \FZ[x]$ be an polynomial of degree $d > 0$ and let $p$ be a prime
which does not divide the coefficient of $x^d$. 
If $r \in \FZ/p\FZ$ is a root of $f(x) \equiv 0 \mod p$ of multiplicity $m$, then
for $n \geq 2$ we have
\begin{equation*}
\left|\left\{ z \in  \FZ/p^n\FZ \; | \; f(z) \equiv 0 \negthickspace\mod p^n , \;
 z \equiv r \negthickspace\mod p \right\} \right| \; \leq \;
p^{n - \lceil \frac{n}{m}\rceil} \; ,
\end{equation*}
where $\lceil a \rceil$ denotes the smallest integer greater than or equal to $a$.
\end{prop}

\begin{proof}
When $m =1$, this is a usual version of Hensel's lemma. 
We assume $m \geq 2$. 
And by replacing $f(x)$ by $f(x+\tilde{r})$ ($\tilde{r} \in \FZ$ being a lift of $r$),
we may assume $r = 0$.
By the polynomial form of Hensel's lemma, there exists a decomposition
$f(x) = g(x)h(x)$ in $\FZ_p[x]$ lifting $f(x) = x^m \bar{h}(x)$ in $\FZ/p\FZ[x]$.
In other words, there exists polynomials $g(x), h(x) \in \FZ_p[x]$ with $f(x) = g(x)h(x)$
such that $g(x)$ is monic, relatively prime to $h(x)$ and $g(x) \equiv x^m \mod p$.

Let $g(x) = g_1(x)^{m_1}\cdots g_k(x)^{m_k}$ be the decomposition as a product
of irreducible polynomials in $\FQ_p[x]$ (we can take them in $\FZ_p[x]$) and
for each $i$, let $\alpha_{i,1},\alpha_{i,2},\cdots, \alpha_{i,d_i}$ 
be the roots of $g_i(x)$ in an algebraic closure $\overline{\FQ_p}$ 
of $\FQ_p$ ($d_i = \deg \; g_i$).
Recall that the nonarchimedean norm $|\;|_p$ on $\FQ_p$ is canonically
extended to $\overline{\FQ_p}$ and the same is true for $\ord_p = - \log_p | \cdot |_p$.
For $\alpha \in \FZ_p$ with $\alpha \equiv 0 \mod p$,
we have $\left|h(\alpha)\right|_p = 1$ and 
\begin{equation*}
\left|f(\alpha)\right|_p  = \prod_{i=1}^k 
\left| (\alpha-\alpha_{i,1})\cdots(\alpha-\alpha_{i,d_i})\right|_p^{m_i} 
= \prod_{i=1}^k | \alpha - \alpha_i |_p^{d_i m_i} \; ,
\end{equation*}
where we have put $\alpha_i = \alpha_{i,1}$.
Hence if $\ord_p f(\alpha) \geq n$, then there exists $1 \leq i \leq k$ with
\begin{equation*}
\ord_p (\alpha - \alpha_i) \geq \frac{n}{d_1 m_1 + \cdots + d_k m_k} = \frac{n}{m} \; .
\end{equation*}
This determines $\alpha$ modulo  $p^{\lceil \frac{n}{m}\rceil}$.
Hence the set $A$ of $\alpha \in p\FZ_p$ satisfying the above inequality
is stable under the translation action of $p^n \FZ_p$ and
$|A/p^n\FZ_p| \leq p^{n - \lceil \frac{n}{m}\rceil}$. 
This completes the proof.
\end{proof}

\begin{cor}\label{cor-app-cong}
Let $f \in \FZ[x]$ be an polynomial of degree $d > 0$ and let $p$ be a prime
which does not divide the coefficient of $x^d$. 
Then we have for $n \geq 2$
\begin{equation*}
\left|\left\{ z \in  \FZ/p^n\FZ \; | \; f(z) \equiv 0 \negthickspace\mod p^n \right\} \right| \; \leq \;
c p^{n - \lceil \frac{n}{l}\rceil} \; \leq \; d p^{n-\lceil \frac{n}{d}\rceil} \; ,
\end{equation*}
where $c$ is the number of distinct roots 
of $f(x) \equiv 0 \mod p$ and $l$ is the maximum of multiplicity of these roots.
\end{cor}

\section{Iterated constant term and multidimensional residue in general coefficient}\label{app-ict}

The goal of this appendix is to introduce the notion of rational function with coefficient of arbitrary commutative ring
and to check definability of iterated residue of a rational function in several variables in the sense of Parshin(\cite{Parshin}). 
The result presented here is fairly direct and we will state brief idea how it works and won't give any formal proof. 

%

Let  $R$ be a commutative ring with $1$. We may identify a polynomial $f(x_1,\ldots,x_n)$ in variables $(x_1,\ldots,x_n)$ with  functions on $\FA^n_R=\spec( R[x_1,\ldots,x_n])$.
The constant term of $f(x_1,\ldots,x_n)$ at $0$ is well-defined by putting $(x_1,\ldots,x_n)=0$. 
Notice this is independent of the choice of the variables. 
The constant term is a well-defined $R$-linear map on $R[x_1,\ldots,x_n]$.  
\begin{equation}
\begin{split}
\ct: R[x_1, \ldots, x_n] &\longrightarrow R \\
f(x_1,\ldots,x_n)=\sum_{I} a_I x^I &\mapsto f(0,\ldots,0) = a_{0,\ldots,0}
\end{split}
\end{equation}
This extends to the ring of formal power series $R[[x_1,\ldots,x_n]]$ since $\ct$ is continuous
w.r.t. $\fm= (x_1,\ldots, x_n)$-adic topology and 
$$
R[[x_1,\ldots,x_n]]=\varprojlim_\fm R[x_1,\ldots,x_n].
$$
It can be rephrased as multivariable version of the Cauchy integral formula:
$$
\ct(f)=a_{0\ldots0}=\res_0 f(x_1,\ldots,x_n) \frac{dx_1}{x_1}\wedge\ldots\wedge\frac{dx_n}{x_n}
$$
This identification may be taken as the definition of the residue at $0$ for polynomials in several variables. 

However, for a rational function, the above definition is not extended as the evaluation map $(x_1,\ldots,x_n)\mapsto 0$ cannot be extended.

Let us define first a rational function with coefficient in a commutative ring $R$.
\begin{defn}
A rational function in variable $x_1,\ldots,x_n$ with coefficient in $R$ is an element of the localized ring of 
$R[x_1,\ldots,x_n]$ by the multiplicative system of nonzero divisor polynomials. We denote the ring of rational 
functions by $R(x_1,x_2, \ldots,x_n)$.
\end{defn}

In one variable case, the  Cauchy integral formula, if it makes sense, can be used to define 
the constant term:
$$
\ct(f) := a_0 = \res_{x=0} f(x) \frac{dx}x.
$$
Furthermore, the coefficients of nonconstant terms are defined by
$$
a_n := \res_{x=0} x^{-n} f(x) \frac{dx}{x}.
$$
Thus we may identify a rational function $f(x)$ with
$$
\sum_{n\in \FZ}  a_n 
 x^n \in R[[x]][x^{-1}].
$$
for $a_n$ defined as above.
This enables us to embed 
\begin{equation}
\begin{split}
R(x) \; ``\hookrightarrow" \; R((x)):= R[[x]][x^{-1}].
\end{split}
\end{equation}
Notice that $\ct$ is defined on $R((x))$ by taking $\sum_{i=-N}^\infty a_i x^i \mapsto a_0$.
If $R$ is a field, every rational function admits Laurent series expansion at $0$.
However if we take the coefficients of  rational functions from a general commutative ring, this need not hold.
We will say a rational function is \textit{admissible} if  it has a Laurent series. 
For an admissible function, the constant term is well-defined by the Cauchy integral formula.
Here we present a condition for a rational function in a single variable to be admissible without proof.
\begin{lem}
A nonzero rational function 
$f(x) = \frac{g(x)}{h(x)}$ is admissible if and only if it has a factorization
$$
f(x) = \frac1{x^n}\frac{a(x)}{u(x)}. 
$$
\end{lem}
with 
\begin{itemize}
\item $a(x)$ polynomial such that $a(0)\ne 0$.
\item $u(x)$ polynomial such that $u(0) \in R^\times$.
\end{itemize}

The above can be stated to a ratio $\frac{g(x)}{h(x)}$ of two power series $g(x), h(x)\in R[[x]]$.
One should be aware that the condition is equivalent to saying
$u(x) \in R[[x]]^\times$.
It is the reason why any rational function with coefficient in a field admits Laurent series expansion.

Now we consider the several variable case. 
This is the Parshin's residue. 
We begin by relating a formal distribution to a rational function:
$$
\iota : f(x_1,\ldots,x_n) \mapsto \sum_{I\in \FZ^n} a_I x^I \in R[[x_1^{\pm},\ldots,x_n^{\pm}]]
$$
Then we define the constant term, through a multivariable version of Cauchy integral formula. 
$$
\ct_\iota (f) = a_{0\ldots0} = \res_{0} \iota\left(f(x_1,\ldots,x_n)\right)\frac{dx_1}{x_1}\cdots\frac{dx_n}{x_n}
$$
Here the image lies in the additive group of formal distributions, where we don't have multiplication of two elements.
Such an embedding is defined for Parshin's point at $0$. It is a tower of 1-dimensional local fields when $R=k$ is 
a field. For simplicity we consider only flags given by hyperplane arrangement.
A (central) hyperplane arrangement is an ordered tuple $\fA = (H_1,H_2,\ldots,H_m)$ of hyperplanes 
$H_i= V(\alpha_i)$ for  linear form $\alpha_i$ defined on $\FA^n$.
For $\fA$ to be associated with a flag of varieties at $0$, $\fA$ needs necessarily \textit{essential}(i.e. 
$m=n$ and $H_i$ are in general position so that $\cap_{i=1}^n H_i = 0$).
Consider the flag of linear spaces associated to $\fA=(H_1,H_2,\ldots,H_n)$:
$$
\flag_\fA=(V_0,V_1,V_2,\ldots,V_{n-1})
$$
where $V_i=V(\alpha_{i+1},\alpha_{i+2},\ldots,\alpha_n)$.
In this way, we define a Parshin point at $0$.

From now on, suppose for simplicity $\alpha_i=x_i$. 
This is not a big deal as we can always change the variable without loss of generality.

When $R=k$ a field, for this Parshin point(ie. the flag of linear subspaces given by the hyperplane arrangement  $\fA= (H_{x_1},\ldots, H_{x_n})$),  we have unique embedding
$$
\iota:k(x_1,\ldots,x_n) \hookrightarrow k((x_1))((x_2))\cdots((x_n)) 
\left(\subset k[[x_1^{\pm},\ldots,x_n^{\pm}]]\right).
$$
This embedding is defined by iterating the following completion at each stage:
$$
k(x_1,\ldots,x_i)
\overset{\iota_i}{\hookrightarrow} k(x_1,\ldots,x_{i-1})((x_{i})).
$$
$\iota_i$ yields the constant map $\ct_{x_i=0}$ w.r.t. $x_i$: 
$$
\ct_{x_i} f(x_1,\ldots,x_i) := \ct_{x_i} \iota_i (f(x_1,\ldots,x_i)) = \res_{x_i=0} \iota(f) \frac{dx_i}{x_i} = a_0(x_1,\ldots,x_{i-1}),
$$
where $\iota_i f(x_1,\ldots,x_i) = \sum_{k=-N}^\infty a_k (x_1,\ldots,x_{i-1})x_i^k$.

The \textit{iterated constant term} w.r.t. $\fA$ is defined by iterating the residue map at each variable $x_i$.
Namely,
\begin{equation}
\begin{split}
\ict_\fA( f(x_1,x_2,\ldots,x_n)) := \ct_{x_1}\circ\ct_{x_1}\circ\cdots\circ\ct_{x_n} f(x_1,x_2,\ldots,x_n) \\
= 
\res_{x_1=0}\circ\cdots\circ\res_{x_n=0} f(x_1,\ldots,x_n)\frac{dx_1}{x_1}\wedge\cdots\wedge\frac{dx_n}{x_n}
\end{split}
\end{equation}

One should note here that the iterated constant term depends heavily on the order of the variables or equivalently the
order of the hyperplanes giving the flag. 
For example, we may consider the two embeddings of $k(x,y)$ into $k[[x^\pm,y^\pm]]$:
\begin{align}
\iota_{x,y}:k(x,y) \to k((x))((y))\\
\iota_{y,x}:k(x,y) \to k((y))((x))
\end{align}
Then 
$$
\iota_{x,y} \left( \frac1{x-y}\right) = \frac1x + \frac{y}{x^2} + \frac{y^2}{x^3} +\ldots,
$$
while
$$
\iota_{y,x} \left(\frac1{x-y}\right) = -\frac1y - \frac{x}{y^2} -\frac{x^2}{y^3}-\ldots.
$$
Thus $\ict_{x,y}\left(\frac1{x-y}\right) = 1$ but  $\ict_{y,x}\left(\frac1{x-y}\right) = -1$. 

Now we consider rational functions in variables $x_1,\ldots,x_n$ with coefficient in a general commutative 
ring $R$:
$$
r(x_1,\ldots,x_n)= \frac{f(x_1,\ldots,x_n)}{g(x_1,\ldots,x_n)}
$$
for $f(x_1,\ldots,x_n), g(x_1,\ldots,x_n)(\ne 0) \in R[x_1,\ldots, x_n]$.
Here we assume that $g(x_1,\ldots,x_n)$ is not a zero-divisor.
\begin{defn}
Let $\fA$ be a Parshin point at $0$.
A rational function with coefficient in a commutative ring is called $\fA$-admissible, if it has a well-defined 
iterated residue(or equivalently iterated constant term) w.r.t. $\fA$. 
\end{defn}

Note that $\fA$-admissibility  is not preserved if the order of hyperplanes defining $\fA$ is changed.  

We are going to check the possibility to define the constant term of $r(x_1,\ldots,x_n)$ w.r.t. $\fA$ as above.
We can rewrite $g(x_1,\ldots,x_n)$ in a unique way:
\begin{equation}
\begin{split}
&g(x_1,\ldots,x_n) \\
=& a_0 + a_1(x_1) x_1 + a_2(x_1,x_2) x_2 + a_3(x_1,x_2,x_3) x_3 +\cdots + a_n(x_1,\ldots,x_{n}) x_n
\end{split}
\end{equation}
for $a_i(x_1,\ldots,x_i)\in R[x_1,\ldots,x_i]$.

Now we check the $\fA$-admissibility of $r(x)$ in terms of the above form of $g(x_1,\ldots,x_n)$.
The simplest case is when $a_0$ is a unit in $R$. 
In this case, $g(x_1,\ldots,x_n)$ is already invertible in $R[[x_1,\ldots,x_n]]$.
Thus $r(x_1,\ldots,x_n)$ lies in $R[[x_1,\ldots,x_n]]$ and is $\fA$-admissible for trivial reason.
The most general case is depicted by the following theorem:
\begin{thm}
Let $r(x_1,\ldots,x_n)=\frac{f(x_1,\ldots,x_n)}{g(x_1,\ldots,x_n)}$ and 
$g(x_1,\ldots, x_n)=a_0 + a_1(x_1) x_1 + a_2(x_1,x_2) x_2 + a_3(x_1,x_2,x_3) x_3 +\cdots + a_n(x_1,\ldots,x_{n}) x_n$.
Then
$r(x_1,\ldots,x_n)$ is admissible 
iff for the smallest $i$ such that $a_i(x_1,\ldots,x_i)$ is not zero-divisor, 
 $a_0, a_1(x_1),  \ldots,a_{i-1} (x_1,\ldots,x_{i-1}) x_{i-1}$ are nilpotent and
$\frac{1}{a_i(x_1,\ldots,x_i)}$ is $\fA$-admissible. 
\end{thm}
It is a direct consequence of the previous theorem which states the condition for a rational function in a variable to be admissible. We simply iterate the 1-variable criterion to the Parshin point to obtain the higher dimensional result.


\begin{thebibliography}{99}

\bibitem{Ap} Apostol, T., {\it Generalized Dedekind sums and transformation formulae of certain Lambert series},
	Duke Math. J. {\bf 17} (1950), 147--157.
	
\bibitem{Asai} Asai, T., 
{\it The reciprocity of Dedekind sums and the factor set for the universal covering group of $SL_2(\FR)$}
Nagoya Math. J., {\bf 37} (1970), pp. 67--80



\bibitem{At} Atiyah, M.,
{\it The logarithm of the Dedekind $\eta$-function}, Math. Ann. {\bf  278} (1987), 335-380.

\bibitem{A-Singer} Atiyah, M. and Singer, I. M., {\it The index of elliptic operators. III}. Ann. of Math. {\bf 87} (1968), 546--604.

\bibitem{B} Beilinson, A., {\it Residues and ad\`eles}, Funksional. Anal. i Prilozhen. {\bf 14} (1980), no.1, 44-45 (Russian). English translation: Funct. Anal. Appl. {\bf 14} (1980), no. 1, 34-35.

\bibitem{Bernt} Bernt, B. C., 
{\it Reciprocity Theorems for Dedekind Sums and Generalizations}, 
Adv. Math. {\bf 23} (1977), 285-316.


\bibitem{BV} Brion, M. and Vergne, M.,
{\it Lattice points in simple polytopes}, 
J. Amer. Math. Soc., {\bf 10} (1997), pp. 371--392



\bibitem{Carlitz} Carlitz, L., {\it Some theorems on generalized Dedekind sums}, Pacific J. Math. {\bf 3} (1953), 513--522.

\bibitem{C-D} Charollois, P. and Dasgupta, S., {\it Integral Eisenstein cocycle of $\GL_n$, I: Sczech's cocycle and p-adic L-functions of totally real fields}, arXiv:1206.3050

\bibitem{Coates-Sinott} Coates, J. and Sinott, W., {\it Integrality properties of the values of partial zeta functions}. Proc. London Math. Soc. (3) 34 (1977), no. 2, 365--384.

\bibitem{Dedekind} Dedekind, R., {\it Erl\"auterungen zu zwei Fragmenten von Riemann}, Riemann's Gesammelte Mathematische Werke, 2nd edition, (1892), 466--472.

\bibitem{Del} Deligne, P., {\it Applications de la formule des
traces aux sommes trigonom\'{e}trique}, in {\it Cohomologie
Etale, S\'{e}minaire de G\'{e}om\'{e}trie Alg\'{e}brique 4 1/2} 
by P. Deligne, Lecture Notes in Math. {\bf 569} (1977), 168-232.

\bibitem{DR} Deligne, P. and Ribet, K., {\it Values of abelian L-functions at negative integers over totally real fields}, Inv. Math. {\bf 59} (3), (1980), 227--286.

\bibitem{DL} Denef, J. and Loeser, F.,
{\it Weights of exponential sums, intersection cohomology, and Newton polyhedra,}
Invent. Math. {\bf 106} (1991), 275-294.

\bibitem{FY} Fukuhara, S. and Yui, N., {\it A generating function of higher-dimensional Apostol-Zagier sums and its reciprocity law},
J. of Number Theory, {\bf 117} (2006) 87--105.

\bibitem{Fulton} Fulton, W.,
{\it Introduction to toric varieties,} Princeton University Press, 1993.

\bibitem{GP} Garoufalidis, S. and Pommersheim, J. E.,
{\it Values of zeta functions at negative integers, Dedekind sums and toric geometry,} 
J. Amer. Math. Soc. {\bf 14} (2000), no.1, 1-23. 


\bibitem{Hi} Hickerson, D.,
{\it Continued fractions and density results for Dedekind sums}, 
J. Reine Angew. Math. {\bf 290} (1977), 113-116.

\bibitem{Hir} Hirzebruch, F.,
{\it The signature theorem: reminiscences and recreation}, in 
{\it Prospects in mathematics}, Ann. of Math. Studies {\bf 70}, 
Princeton University Press (1971), 3-31.

\bibitem{HZ} Hirzebruch, F. and Zagier, D.,
{\it The Atiyah-Singer theorem and elementary number theory},
Mathematics Lecture Series {\bf 3}, Publish or Perish, Inc., Boston, MA, 1974.

\bibitem{IK} Iwaniec, H. and Kowalski, E.,
{\it Analytic Number Theory},  Amer. Math. Soc. 2004.

%

\bibitem{J-L4} Jun, B. and Lee, J.,
{\it Special values of partial zeta functions of real quadratic fields at nonpositive integers and Euler-Maclaurin formula}, preprint (2012), arXiv:1209.4958, 39pages.
 
\bibitem{J-L2} Jun, B. and Lee, J.,
{\it Equidistribution of generalized Dedekind sums and exponential sums}, J. Number Theory {\bf 137} (2014), 67--92.

\bibitem{K-M} Kirby, R. and Melvin, P.,
{\it Dedekind sums, $\mu$-invariants and the signature cocyle},
Math. Ann. {\bf 299} (1994), 231-267.


\bibitem{Manin}
Manin, Yu. I.,
{\it Parabolic points and zeta-functions of modular curves}, 
Math. USSR Izv., {\bf 6} (1) (1972), pp. 19--64,
Selected Papers, World Scientific (1996), pp. 202--247

\bibitem{Myerson} 
Myerson, G.,
{\it Dedekind sums and uniform distribution},
J. Number Theory, {\bf 28} (1988), pp. 233-239

\bibitem{Me} Meyer, C.,
{\it \"Uber die Berechnung der Klassenzahl abelscher K\"orper quadratischen 
Zahlk\"orpern,} Akademie-Verlag Berlin, 1957.

\bibitem{Parshin} Parshin, A. N., {\it On the arithmetic of two-dimensional schemes. I. Distributions and residues,} 
Izv. Akad. Nauk SSSR Ser. Mat. {\bf 40} (1976), no. 4, 736--773 (Russian).

\bibitem{Pom2} Pommersheim, J. E.,
{\it Barvinok's algolithm and the Todd class of a toric variety,} 
J. Pure Appl. Alg. {\bf 117, 118} (1997), 519-533.

\bibitem{RG} Rademacher, H. and Grosswald, E.,
{\it Dedekind Sums},
Carus Math. Monogr., vol. 16 Math. Assoc. Amer. (1972)

\bibitem{Sczech} Sczech, R., {\it Dedekind sums and signatures of intersection form}, 
Math. Ann., {\bf 299} (1) (1994), 269--274.

\bibitem{Selberg}  Selberg, A., {\it
On the Estimation of Fourier Coefficients of Modular Forms},
Proc. Sympos. Pure Math., vol. VIIIAmer. Math. Soc., Providence, RI (1965)

\bibitem{Siegel} Siegel, C.~L., {\it Bernoullische Polynome und quadratische Zahlk\"orper}, Nachr. Akad. Wiss. G\"ottingen Math.-physik {\bf 2} (1968), 7--38.

\bibitem{Shintani} Shintani, T.,
{\it On special values of zeta functions of totally real algebraic number fields
at non-positive integers,} J. Fac. Sci. Univ. Tokyo. {\bf 63} (1976), 393-417.

\bibitem{Solomon} Solomon, D.,
{\it Algebraic properties of Shintani's generating functions: Dedekind sums and cocycles
on $\operatorname{PGL}_2(\FQ)$,} Compositio Math. {\bf 112} (1998), no.3, 333-362.

\bibitem{Stevens} Stevens, G., {\it The Eisenstein measure and real quadratic fields},
Theorie des nombres (Quebec, PQ, 1987), 887--927, de Gruyter, Berlin, 1989.

\bibitem{Szenes} Szenes, A.,
{\it Residue theorem for rational trigonometric sums and Verlinde's formula,}
Duke Math. J. {\bf 118} (2003), no.2, 189-227.

\bibitem{Vardi} Vardi, I.,
{\it A relation between Dedekind sums and Kloosterman sums}, 
Duke Math. J. {\bf 55} (1987), no.1, 189-197.

\bibitem{Weyl} Weyl, H.,
{\it \"Uber die Gleichverteilung von Zahlen mod. Eins},
Math. Ann., {\bf 77} (3) (1916), pp. 313--352.


\bibitem{Zagier1} Zagier, D.,
{\it Higher dimensional Dedekind sums,} Math. Ann. {\bf 202} (1973), 149--172.


\bibitem{Zagier2} Zagier, D., {\it Valeurs des fonctions zeta des corps quadratiques reels aux entiers negatifs,}
Asterisque {\bf 41-42} (1977), 135--151
\end{thebibliography}
\end{document}